\documentclass[11point]{amsart}
\usepackage{amsmath, xcolor}
\usepackage{amscd}

\usepackage{amssymb,amsmath,amscd,epsfig,amsthm,enumerate,import,graphicx}
\usepackage{caption}
\usepackage{subcaption}

\usepackage{geometry}                
\newtheorem{theorem}{Theorem}[section]
\newtheorem{proposition}[theorem]{Proposition}
\newtheorem{lemma}[theorem]{Lemma}
\newtheorem{corollary}[theorem]{Corollary}
\theoremstyle{definition}
\newtheorem{definition}[theorem]{Definition}

\theoremstyle{remark}
\newtheorem{remark}[theorem]{Remark}

\newcommand{\be}{\begin{equation}}
\newcommand{\ee}{\end{equation}}

\makeatletter
\newcommand{\tpitchfork}{%
  \vbox{
    \baselineskip\z@skip
    \lineskip-.52ex
    \lineskiplimit\maxdimen
    \m@th
    \ialign{##\crcr\hidewidth\smash{$-$}\hidewidth\crcr$\pitchfork$\crcr}
  }%
}
\makeatother

\newcommand{\D}{{\bf D}}

\newcommand{\R}{{\bf R}}

\newcommand{\s}{{\bf s}}

\newcommand{\zbar}{\overline{z}}
\newcommand{\wbar}{\overline{w}}

\newcommand{\bbC}{{\mathbb C}}

\newcommand{\bbN}{{\mathbb N}}
\newcommand{\bbR}{{\mathbb R}}
\newcommand{\bbH}{{\mathbb H}}

\newcommand{\calV}{{\mathcal V}}
\newcommand{\calD}{{\mathcal D}}
\newcommand{\calB}{{\mathcal B}}

\newcommand{\calU}{{\mathcal U}}

\newcommand{\calK}{{\mathcal K}}

\newcommand{\calF}{{\mathcal F}}
\newcommand{\calC}{{\mathcal C}}
\newcommand{\calR}{{\mathcal R}}

\newcommand{\calW}{{\mathcal W}}
\newcommand{\calN}{{\mathcal N}}
\newcommand{\calO}{{\mathcal O}}

\newcommand{\calI}{{\mathcal I}}

\newcommand{\calT}{{\mathcal T}}

\newcommand{\tr}{{\mbox{Tr}}}

\newcommand{\dbar}{\bar\partial}

 \newcommand\floor[1]{\lfloor#1\rfloor}

\newcommand{\inner}[2]{\langle#1\,,#2\rangle}
\newcommand{\norm}[1]{\| #1\|}

\newcommand{\ket}[1]{| #1\rangle}

\renewcommand{\Box}{\square}

\reversemarginpar

  \DeclareMathOperator{\Det}{Det}
  
  \DeclareSymbolFont{bbold}{U}{bbold}{m}{n}
\DeclareSymbolFontAlphabet{\mathbbold}{bbold}

\newcommand{\1}{\mathbbold{1}}
  
  \newcommand{\Sz}{Szeg\"o\ }
 
 \newcommand{\C}{\bbC}
 \newcommand{\N}{\bbN}
  \newcommand{\barw}{\wbar}

\begin{document}

\title{Szeg\"o limit theorems for singular Berezin-Toeplitz operators}
\author{Salvador P\'erez-Esteva}
\thanks{S. P-E. partially supported by the Mexican Grant PAPIIT-UNAM IN106418}
\address{Instituto de Matem\'aticas\\
Unidad Cuernavaca\\
Universidad Nacional Aut\'onoma de M\'exico}
\email{spesteva@im.unam.mx}
\author{Alejandro Uribe}
\thanks{}
\address{Mathematics Department\\
University of Michigan\\Ann Arbor, Michigan 48109}
\email{uribe@umich.edu}

\date{}
\begin{abstract}
We consider Berezin-Toeplitz operators whose multipliers are compactly
supported densities carried by 
a submanifold $\Gamma\subset\bbC^N$.  We compute asymptotically the moments
of their spectral measures, and we prove Szeg\"o limit theorems
in cases when $\Gamma$ is isotropic or
co-isotropic, from which Weyl estimates follow.  
We also obtain asymptotics of the Schatten norms of such
operators. Rescaled versions of these operators can be thought of as
quantum mechanical mixed states, and our results give the semi-classical limit of 
their entropy.
\end{abstract}
\newcommand{\constants}{\left(\frac{k}{\pi}\right)^N}


\maketitle

\centerline{Dedicated to Victor Guillemin on his 80\textsuperscript{th} Birthday}
\bigskip

\tableofcontents

\section{Introduction and statements of the main results}

We will consider operators on the weighted Bargmann space
\begin{equation}\label{}
\calB_k = \left\{ \psi\in L^2(\bbC^N, dL)\;;\; \psi(z) = f(z)e^{-k|z|^2/2},\ \dbar f = 0\right\}
\end{equation}
where $dL$ denotes Lebesgue measure.

\medskip
Let $\Gamma\subset\bbC^N\cong\bbR^{2N}$ be a smooth submanifold.
Although some of our examples of $\Gamma$ are non-compact, we will only
be considering symbols or amplitudes on $\Gamma$ that are compactly supported.  
 
\medskip
As a motivation for the type of operators we consider in this article, let $d\sigma$
be the measure induced on $\Gamma$ by Lebesgue measure and $L^2(\Gamma)$
the corresponding $L^2$ space.  Let
\[
\calR_k: \calB_k \to L^2(\Gamma)
\]
be the restriction operator, which is bounded, and let
\[
\calR_k^* :L^2(\Gamma)\to \calB_k 
\]
be its adjoint.  Then the self-adjoint operator
\begin{equation}\label{tGamma}
T_\Gamma:= \calR^*\circ\calR: \calB_k\to \calB_k
\end{equation}
is an example of what 
we will call a singular Toeplitz operator, since it can be considered as a Toeplitz
operator with a multiplier that is a distribution supported on $\Gamma$, as the following lemma
shows:
\begin{lemma}
Let $\Pi_k(z,\wbar)$ be the reproducing kernel for $\calB_k$, that is,
the Schwartz kernel of the projection $\Pi_k L^2(\bbC)\to\calB_k$.  
Then
\[
\forall \psi\in\calB_k\qquad T_\Gamma(\psi)(z) = \int_\Gamma \Pi_k(z,\wbar)\,\psi(w)\, d\sigma(w).
\]
\end{lemma}
\begin{proof}
We compute $\calR^*$.  Let $\varphi\in L^2(\Gamma)$ and $\psi\in \calB_k$.  Then
\[
\inner{\calR^*(\varphi)}{\psi} = \inner{\varphi}{R(\psi)} = \int_\Gamma \varphi(\zeta) \overline{\psi(\zeta)}\,
d\sigma(\zeta).
\]
Now use the reproducing property:
\[
\psi(\zeta) = \int_{\bbC^N} \Pi_k(\zeta,\wbar)\psi(w)\, dL(w),
\]
and using Fubini's theorem we get that
\[
\inner{\calR^*(\varphi)}{\psi} =  \int_{\bbC^N} \overline{\psi(w)}\left(
\int_\Gamma \varphi(\zeta) \overline{\Pi_k(\zeta,\wbar)}\, d\sigma(\zeta)\right)\,dL(w).
\]
Since this is true for all $\psi$,
\begin{equation}\label{rstar}
\calR^*(\varphi) (w)= \int_\Gamma \varphi(\zeta) \Pi_k(w,\overline{\zeta})\, d\sigma(\zeta),
\end{equation}
as desired.
\end{proof}

\medskip
More generally, in this paper we  consider the following operators:
\begin{definition}
Let $a:\Gamma\to\bbC$ be a smooth function of compact support.  
Then $T_{ad\sigma}:\calB_k\to\calB_k$
is the operator
\[
T_{ad\sigma}(\psi) (z) = \int_\Gamma \Pi_k(z,\wbar)\,\psi(w)\, a(w)\, d\sigma(w).
\]
\end{definition}

The reproducing kernel for $\calB_k$ is
\begin{equation}\label{repro}
\Pi_k(z,\wbar) 
= \constants\, e^{k\,z\wbar}\,e^{-k|z|^2/2}\,e^{-k|w|^2/2}\\
= \constants\, e^{-k\,|z-w|^2/2}\, e^{ik\,\omega(z,w)},
\end{equation}
where $\omega$ is the symplectic form
\begin{equation}\label{sForm}
\omega(z,w)= \frac{1}{2i}\left(z\wbar - \zbar w\right).
\end{equation}
Therefore we have the explicit formula 
\begin{equation}\label{}
T_{ad\sigma}(\psi) (z) = \constants e^{-k|z|^2/2}\,\int_\Gamma e^{k\,z\wbar}\,e^{-k|w|^2/2}\,
\psi(w)\, a(w)\, d\sigma(w).
\end{equation}
\bigskip

In the  theory of Bergman spaces, one considers the holomorphic parts of the functions in $\calB_k$, that is,  the Fock space $B_k$ for $k>0,$   consisting of all entire functions $f$ on $\C^N$ such that $f\in L^2(\C^N,dL_k)$, where 
\begin{equation*}
dL_k=\left(\frac{k}{\pi}\right)^N e^{-k\vert z\vert^2}dL.
\end{equation*}

 It is clear that the multiplication operator  $M_kf(z)=e^{-k\vert z\vert^2/2}f(z)$ is an isomorphism of $B_k$ onto $\calB _k$.
 The operator  $M_k^{-1}T_{ad\sigma}M_k$ is the Toeplitz operator on $B_k$ \begin{equation*}
\calT_{ad\sigma}f(z)=\int_{\Gamma}e^{kz\wbar}f(w)\left(\frac{k}{\pi}\right)^N e^{-k\vert w\vert^2}d\sigma(w). 
\end{equation*}
Therefore our results have direct translations to corresponding Toeplitz operators in the Fock space.
We chose to work on $\calB_k$ for convenience.

\bigskip
A second motivation for the study of these operators comes from quantum mechanics. 
for each $w\in\bbC^N$ consider the {\em coherent state at $w$}, $e_w\in\calB_k$:
\[
e_w(z) = \Pi_k(z,\wbar).
\]
The orthogonal projection $P_w$ onto the line spanned by $e_w$ is
\begin{equation}\label{equivalently}
P_w(\psi)(z) = \frac{1}{\Pi_k(w,\wbar)}\, \psi(w)\,\Pi_k(z, \wbar) = \psi(w)\, e^{kz\wbar}\,e^{-|w|^2/2}
\, e^{-|z|^2/2},
\end{equation}
from where it follows that
\begin{equation}\label{superProj}
T_{ad\sigma} = \constants \int_\Gamma a(w)\,P_w\,d\sigma(w).
\end{equation}
In other words, up to an overall normalization factor,
$T_{ad\sigma}$ is a weighted superposition of the rank-one projectors
$P_w$ over $\Gamma$.  From the expression (\ref{superProj}) it follows that
\begin{equation}\label{trace}
\tr(T_{ad\sigma}) = \constants \int_\Gamma a(w)\, d\sigma(w).
\end{equation}
If $a>0$, the operator
\begin{equation}\label{}
\rho_a := \frac{1}{\tr(T_{ad\sigma})}\, T_{ad\sigma}
\end{equation}
is non-negative and of trace one.  In the language of quantum mechanics $\rho$
is a {\em mixed state} or a {\em density matrix}.

\medskip
Mixed states of this kind (with $\Gamma$ Lagrangian) 
have been studied by Yohann Le Floch in \cite{LF}, with $\bbC^N$ 
replaced by a compact quantized K\"ahler manifold.  
(This work of Le Floch focuses on estimating the so-called fidelity of such 
states.)
Quantum mechanically, they are partial traces of a pure state shared by the Bargmann
space and by $L^2(\Gamma)$.  In a standard interpretation, by taking a partial trace over $L^2(\Gamma)$, the latter is acting 
 as the host of the (unknown) background state space.   If $\{\psi_j\}$ is an orthonormal
 basis of eigenfunctions of $\rho_a$ with eigenvalues $p_j$ (so that $p_j\geq 0$ and
 $\sum_j p_j=1$), a possible quantum mechanical interpretation of $\rho_a$ is that
$\forall j$ it represents the state $\psi_j$ with probability $p_j$, see \cite{BH} \S 19.3.
 Since the greatest eigenvalue of $T_\Gamma$ is
 \begin{equation}\label{}
\lambda_{\max}(k) = \sup_{\psi\in\calB_k\setminus\{0\}}
\frac{\int_\Gamma \left|\psi(z)\right|^2\,d\sigma(z)}{\norm{\psi}_{\calB_k}^2}
\end{equation}
(as follows from (\ref{tGamma})),
the  the eigenvector corresponding to the greatest probability $p_j$ 
is the function $\psi\in\calB_k$ that is most concentrated on $\Gamma$ in the $L^2$ sense.

Incidentally, the case $\Gamma$ Lagrangian is special in that the restriction operator
$\calR$ is injective since $\Gamma$ is then totally real and middle-dimensional.

\bigskip
In this paper we study the asymptotics of the spectrum of $T_{ad\sigma}$ and related
operators in the semi-classical limit, that is, as $k\to\infty$.  We will obtain the asymptotic
behavior of the moments of their spectral measure, and \Sz limit theorems of normalizations 
of these operators for special classes of submanifolds $\Gamma$.   We now 
turn to stating our main results.

\subsection{The norm estimate}
Our first result is an upper bound on the operator norm of $T_{ad\sigma}$: 

\begin{theorem}\label{NormEst}  Let $\Gamma$ be a smooth  manifold in  
$\bbC^N$ of dimension $d\leq 2N$ 
and $a\in L^{\infty}(\Gamma)$ of compact support. Then there exists $C>0$ such that the operator norm
of $T_{ad\sigma}$ satisfies
\begin{equation}\label{estima}
 \Vert T_{ad\sigma}\Vert_{\mathcal{L}B_k} \leq C\Vert a\Vert_{\infty}\,k^{N-d/2}.
 \end{equation} 
 \end{theorem}
 As we will see this estimate is sharp in general.  A
 Szeg\"o limit theorem applies to a family of operators whose spectra are contained
 on a fixed interval, and this result sets the dependence on $k$ of the normalization
one needs to make on the $T_{ad\sigma}$ in order to obtain such a theorem.

\subsection{The trace of a multiple composition}
Next we state a general theorem on the asymptotics of the composition of $n$
operators of the form $T_{ad\sigma}$ associated to an arbitrary submanifold 
 $\Gamma\subset\bbC^N$. 
 
To state our result we need to introduce an endomorphism of the tangent bundle of 
$\Gamma$ which we denote by $K:T\Gamma\to T\Gamma$ and is
defined as follows.  Let $\Pi_w: \bbR^{2N}\to T_w\Gamma$ be the
orthogonal projection.  Then $K_{w}: T_{w}\Gamma\to T_{w}\Gamma,$ is defined by
\begin{equation}\label{}
K_w := \Pi_w\circ J
\end{equation}
where $J:\bbR^{2N}\to\bbR^{2N}$ is the complex structure such that
\begin{equation}\label{}
\omega(v,w) = v\cdot J(w).
\end{equation}
Specifically,  the right-hand side is the dot product:  In complex notation
$u\cdot v = \Re\left(\sum_{j=1}^N u_j\,\overline{v_j}\right)$, and, in accordance with
(\ref{sForm}),  $J$ is multiplication by $\sqrt{-1}$.

One can say that $K$ is the projection of $J$ on the tangent spaces of $\Gamma$.
$K_w$ is skew-adjoint with respect to the Euclidean inner product on $T_w\Gamma$, and
therefore its eigenvalues are purely imaginary and come in conjugate pairs.  Let us write 
the non-zero eigenvalues with multiplicities as
\[
\pm i \lambda_\ell,\quad 0< \lambda_1\leq\cdots \leq \lambda_r.
\]
Here $r$ is half the rank of  $K_w$. 
In general the $\lambda_\ell$ and $r$ depend on the base point $w\in \Gamma$, but we have
suppressed that dependence from the notation, for simplicity.

\smallskip
We can now state:
\begin{theorem}\label{ManyOps}
Let $a_j\in C_0^\infty(\Gamma, \bbC)$, $j=1,\ldots , n$,
and consider the composition
 \[
 \Upsilon := T_{a_1d\sigma}\circ\cdots\circ T_{a_nd\sigma}.
 \]
 Then
\begin{equation}\label{laMeraPapa}
\tr\left(\Upsilon \right) = 
\left[ 2^{d/2}\left(\frac{k}{\pi}\right)^{N-d/2}\right]^n\, \left(\frac{k}{2\pi}\right)^{d/2}\ 
\left(\int_{\Gamma}{\frac{1}\Delta_n}\prod_{j=1}^n a_j(w)\, d\sigma(w) + O(1/k)\right)
\end{equation}
where
\begin{equation}\label{deltaene}
\Delta_n = n^{\frac{d}{2}-r}
\prod_{\ell = 1}^r\,
\frac{(1+\lambda_\ell)^{n} - (1-\lambda_\ell)^n}{2\lambda_\ell} .
\end{equation}
\end{theorem}
If we specialize to the case $a_1 = a_2 = \cdots = a_n = a$, this theorem gives us the
asymptotics of the moments of the spectral measure of $T_{ad\sigma}$.

\smallskip
This theorem is proved by the method of stationary phase, and
$\Delta_n$ arises as a factor of the square root of the Hessian of the phase.  
$\Delta_n$ will be constant (with respect to $w\in \Gamma$) and explicit in       
the special cases we will consider below.  Its dependence on $n$ however 
introduces a new feature with respect to the ``standard" \Sz limit theorem,
which is the previous theorem in the case $\Gamma = \bbC^N$:

\begin{corollary}
Let $a\in C_0^{\infty}(\bbC^N, \bbC)$.  Then for each $n=1,2,\cdots$
\[
\tr (T_{adL}^n) = \left(\frac{k}{\pi}\right)^N\,
\left(\int_{\bbC^N} a(w)^n\, dL(w) + O(1/k)\right).
\]
\end{corollary}
\begin{proof}
Since $d=2N$, the constant in front of the integral in (\ref{laMeraPapa}) evaluates to
$2^{N(n-1)}\, \left(\frac{k}{\pi}\right)^N.$
On the other hand in this case $K=J$, so $r=N$ and all the $\lambda_\ell$ are equal to one.
Therefore $\Delta_n = 2^{N(n-1)}$, and the powers of $2$ cancel each other out.
\end{proof}

This Theorem is stated in \cite{BMS} with $\bbC^N$ replaced by a compact K\"ahler
manifold.  It is a ``folk" theorem in the $\bbC^N$ case, and it clearly holds for 
more general amplitudes $a$, for example for $a$ in the Schwartz class. 
We were not able to find a reference for it in the literature.

\medskip
More generally, in case $\Gamma$ is a complex submanifold of real dimension $d=2r$,
all $\lambda_\ell$ are equal to one which implies  (by a similar argument as before) that
\[
\Delta_n = 2^{r(n-1)},
\]
and Theorem \ref{ManyOps} in turn implies that
\[
\tr \left(\left(\frac{\pi}{k}\right)^{N-r}T_{ad\sigma}\right)^n = \left(\frac{k}{\pi}\right)^r\,
\left(\int_{\Gamma} a(w)^n\, dL(w) + O(1/k)\right).
\]
This indicates that $T_{ad\sigma}$ is really a standard Berezin-Toeplitz operator
on a K\"ahler manifold of real dimension $d=2r$ except that it has a large kernel, 
corresponding to all holomorphic functions vanishing on $\Gamma$.

\subsection{The \Sz limit theorem}

We now specialize Theorem (\ref{ManyOps}) to certain classes of submanifolds $\Gamma$,
which will allow us to obtain finer results.

We begin by recalling some notions from 
symplectic geometry.  Given a subspace $V\subset\bbR^{2N}\cong\bbC^N$, let us denote
its symplectic annihilator by
\[
V^\circ = \{x\in \bbR^{2N}\;;\; \forall v\in V\ \omega(x,v)=0\}.
\]
If we denote by $V^\bot$ the subspace perpendicular to $V$ with respect to the Euclidean
inner product,  it is easy to check that
\[
V^\bot = J\left(V^\circ\right),
\]
and conversely.  It follows that if we define $K:V\to V$ as above, namely
$K = \Pi_V\circ J, \  \Pi_V:\bbR^{2N}\to V\ \text{orthogonal projection},$
then
\begin{equation}\label{usefulFact}
\ker(K) = V^\circ \cap V.
\end{equation}
This will be a useful fact.

\begin{definition}
 $V$ is isotropic
(resp.\ co-isotropic, resp.\ Lagrangian) if and only if 
$V\subset V^\circ$ (resp.\ $V^\circ\subset V$ resp.\ $V^\circ = V$).
A submanifold $\Gamma$ is isotropic (resp.\ co-isotropic, resp.\ Lagrangian) if and only if 
$\forall w\in \Gamma$ the tangent space $T_w\Gamma$ is
isotropic (resp.\ co-isotropic, resp.\ Lagrangian).   
\end{definition}
By the skew symmetry of $\omega$ it is automatic that any curve is isotropic
and any hypersurface is co-isotropic.

\bigskip
We now assume that $\Gamma$ is isotropic or co-isotropic.  The \Sz limit
theorem in both cases can be stated at the same time, if we introduce
the following

\medskip\noindent
{\bf Notation:}  Let $d$ be the dimension of $\Gamma$.  Then we let
\begin{equation}\label{}
d' := d\quad\text{if}\ \Gamma\ \text{is isotropic}, \ \text{and}\ 
d' := 2N-d\quad\text{if}\ \Gamma\ \text{is co-isotropic.}
\end{equation}

\smallskip
In order to obtain a \Sz limit theorem we have to normalize $T_{ad\sigma}$ as follows:  
In both the isotropic and co-isotropic cases, define
\begin{equation}\label{}
S_{ad\sigma}:= 2^{-d'/2}\, \left(\frac{\pi}{k}\right)^{N-d/2}\,T_{ad\sigma}.
\end{equation}
By the norm estimate $\norm{S_{ad\sigma}} = O(1)$.

\bigskip
We will prove that Theorem \ref{ManyOps} implies the following:
\begin{theorem}\label{Szego}
Let $\Gamma$ be isotropic or co-isotropic,
and $a\geq 0$ a smooth function of compact support on $\Gamma$.
Let $R>0$ such that
\begin{itemize}
\item[a)] $\sigma(S_{ad\sigma})\subset [0,R]$ for all $k$, and
\item[b)] $a(\Gamma)\subset [0,R]$.
\end{itemize}
Let  $\varphi$  be any function such that for some $p>0$ $\varphi(t) /t^p\in C[0,R]$.
Then
\begin{equation}\label{szegoLimit}
\lim_{k\rightarrow\infty}2^{d'/2}\,\left(\frac{\pi}{k}\right)^{d/2}\tr(\varphi(S_{ad\sigma}))=
\int_\Gamma \calO_{-d'/2}(\varphi)(a(w))\, d\sigma(w),
\end{equation}
where $\calO_{-\alpha}$ is the operator
\begin{equation}
\calO_{-\alpha}(\varphi)(t):=
\frac{1}{\Gamma(\alpha)}\int_0^{t}\varphi(s)\log(t/s)^{\alpha-1}\frac{ds}{s}
\end{equation}
($\Gamma(\alpha)$ is the gamma function evaluated at $\alpha$)
if $\alpha >0$, and $\calO_0$ is the identity.
\end{theorem}

\begin{remark}
If $\varphi$ is a polynomial without constant term the restriction $a\geq 0$ can be lifted,
see Corollary \ref{TracePowers}
\end{remark}

\medskip
In case $a\equiv 1$ the previous Theorem simplifies to:
\begin{corollary}
For any function $\varphi$ as in Theorem \ref{Szego}, if $d'>0$ and $\Gamma$ is compact
one has
\begin{equation}\label{szegoLimit1}
\lim_{k\rightarrow\infty}2^{d'/2}\,\left(\frac{\pi}{k}\right)^{d/2}\tr(\varphi(S_{d\sigma}))=
\frac{|\Gamma|}{\Gamma(d'/2)}\int_0^{1}\varphi(s)\, \left(-\log(s)\right)^{\frac{d'}{2}-1}\,
\frac{ds}{s}
\end{equation}
where $|\Gamma| = \int_\Gamma d\sigma$ is the volume of $\Gamma$.
\end{corollary}

\medskip
\begin{remark}
The operators $\calO_{-\alpha}$ satisfy, for $p>0 $, that
\begin{equation}\label{laPropriedad}
\calO_{-\alpha}(s^p)(t) = \frac{t^p}{p^\alpha}.
\end{equation}
For $\alpha>0$ this follows from the change of variables $x=p\log(t/s)$:
\[
\frac{1}{\Gamma(\alpha)}\int_0^t s^{p}\log(t/s)^{\alpha-1}\frac{ds}{s} =
\frac{1}{\Gamma(\alpha)}\, \int_0^\infty t^p\,e^{-x}\left(\frac{x}{p}\right)^{\alpha-1}\, \frac{dx}{p}
 = \frac{t^p}{p^\alpha}.
\]
\end{remark}

\begin{remark} 
If $d'>0$,  the right-hand side of (\ref{szegoLimit}) equals
\begin{equation}
\frac{1}{\Gamma(d'/2)}\int_{\Gamma}\left(\int_0^{a(w)}\varphi(s)\, \log\left(\frac{a(w)}{s}\right)^{\frac{d'}{2}-1}\,
\frac{ds}{s}\right) d\sigma(w).
\end{equation}
By Fubini's theorem
we can also express this 
as the integral of $\varphi$ with respect to an absolutely continuous
measure $d\nu_a = \calD_a\, ds$ on $[0,R]$, 
\[
\lim_{k\rightarrow\infty}2^{d'/2}\,\left(\frac{\pi}{k}\right)^{d/2}\tr(\varphi(S_{ad\sigma}))=
\int_0^{\max (a)} \varphi(s)\, \calD_a(s)\, ds
\]
with a density given a.e.~ by
\begin{equation}\label{szegoLimit2}
\calD_a(s) = 
\frac{1}{\Gamma(d'/2)\,s}\,
\int_{\{w\in\Gamma\;;\; a(w)\geq s\}}\log\left(\frac{a(w)}{s}\right)^{\frac{d'}{2}-1}d\sigma(w).
\end{equation} 
\end{remark}

This density is supported in $[0, \max(a)]$ and is bounded in any
closed interval $I\subset (0, \max(a)]$ if $d'\geq 2$.
In case $d'=1$ it can be shown that $\calD_a$ is continuous at regular values of $a$.
The graphs of $\calD_a$ in case $a\equiv 1$ for $d'=1$ and $d'=4$ appear in the figure
below.  For all values of $d'$ the density of eigenvalues concentrates at zero, which
is expected by compactness of the operators.  For  $d'=1$ the density of eigenvalues concentrates
at $s=\max(a)$ as well (albeit it temains integrable at $\max(a)$).  
The case $d'=2$ is particularly simple, as the log function disappears.

\begin{figure}[h!]
\centering
\begin{subfigure}{.5\textwidth}
  \centering
  \includegraphics[width=.6\linewidth]{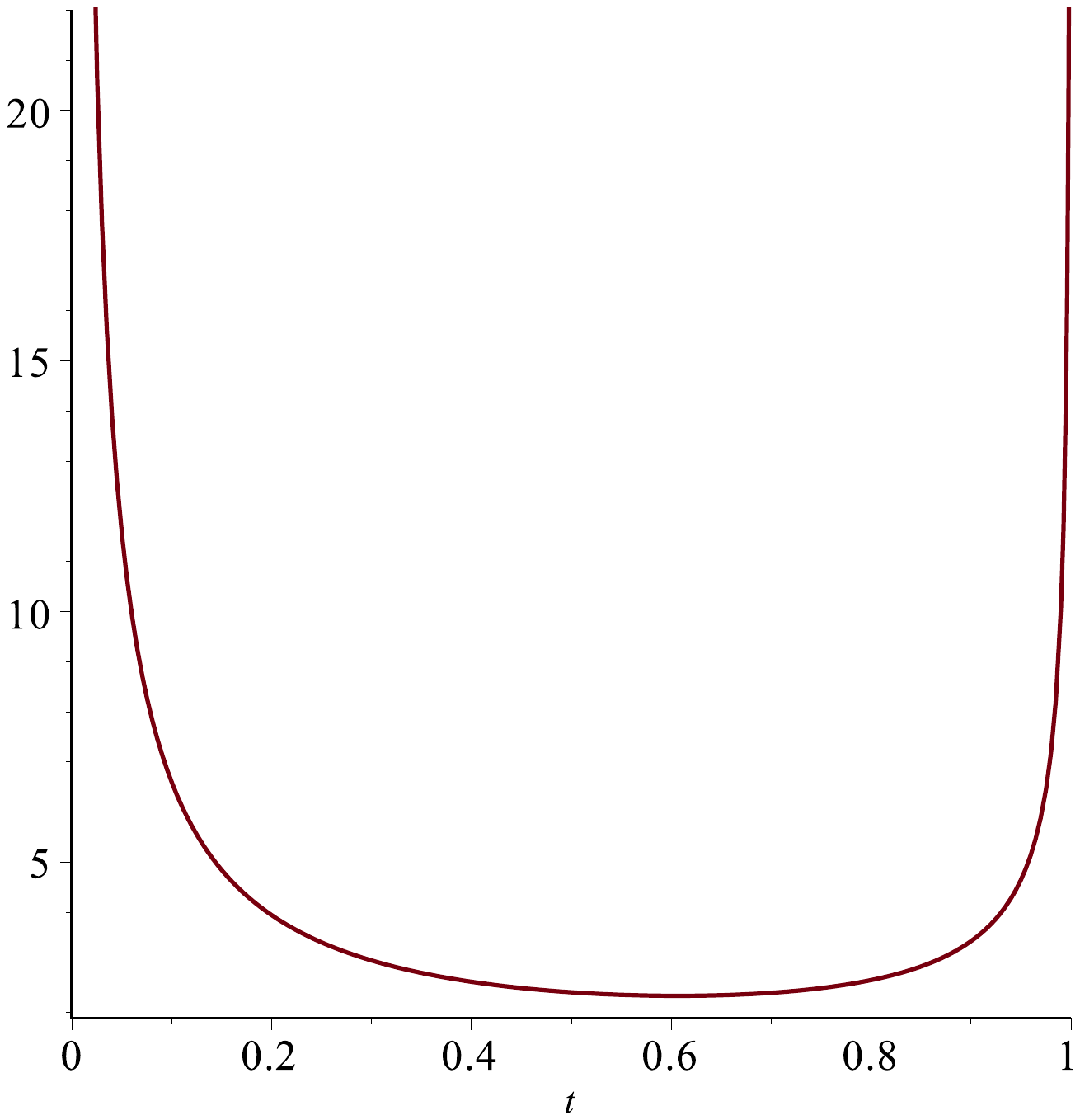}
  \vspace{-2cm}
  \caption{$d'=1$}
  \label{fig:sub1}
\end{subfigure}%
\begin{subfigure}{.5\textwidth}
  \centering
  \includegraphics[width=.6\linewidth]{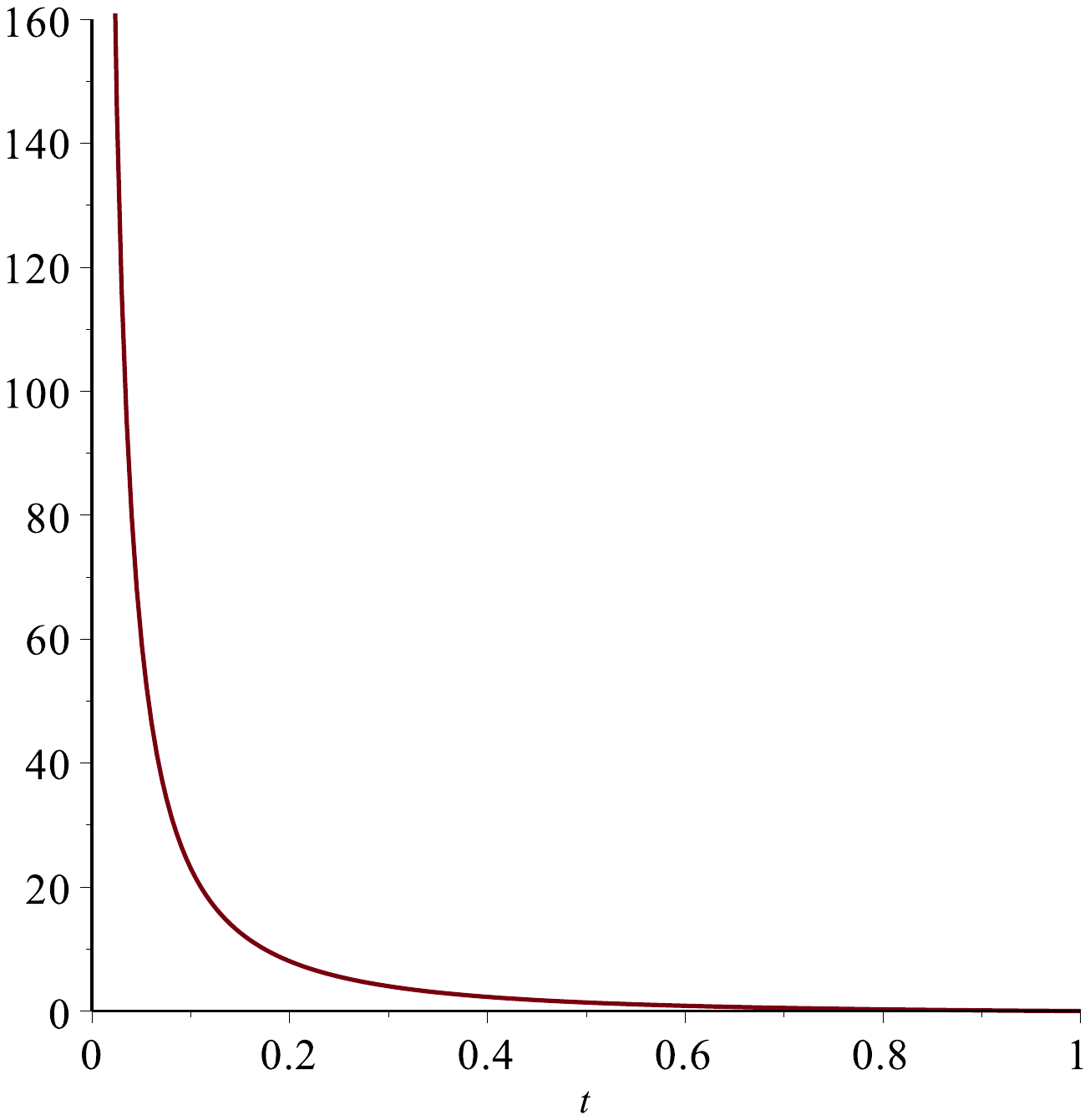}
    \vspace{-2cm}
  \caption{$d'=4$}
  \label{fig:sub2}
\end{subfigure}
\caption{The density $\calD_1$ for $d'=1,\,4.$}
\label{fig:test}
\end{figure}

\begin{remark}
The previous Theorem includes the extreme co-isotropic
case $\Gamma = \bbC^N$.  Since $\calO_0$ is the identity, the limit of 
the spectral measure of $S_{adL}$ is just the push-forward by $a$ of the
volume form on $\bbR^{2N}$.
A similar statement can be derived in the case when $\Gamma$ is a complex submanifold,
but we will not write it explicitly.
\end{remark}

\subsection{Corollaries and further results}

\subsubsection{Weyl estimates}   
As we wil see Theorem \ref{Szego} implies the following Weyl estimate:
\begin{proposition}
Under the hypothesis of Theorem \ref{Szego}, let $I$ denote a closed interval
contained in $(0, \max(a)]$.  Let $N_I(k)$ denote the number of eigenvalues
of $S_{ad\sigma}$ contained in $I$, counted with multiplicities. Then
\[
\lim_{k\to\infty} 2^{d'/2}\,\left(\frac{\pi}{k}\right)^{d/2}\,N_I(k) = \int_I\calD_a(s)\, ds.
\]
\end{proposition}
\begin{proof}
 $N_I(k) = \tr\left[\1_I(S_{ad\sigma})\right]$ where $\1_I$ is the characteristic function of $I$.
It is easy to construct two sequences of trapezoidal functions $\{\varphi_n\}$ and $\{\psi_n\}$ such that
\[
\forall n \qquad \varphi_n\leq \1_I \leq \psi_n
\]
and such that $\lim_{n\to\infty} \varphi_n = \1_I = \lim_{n\to\infty}\psi_n$ pointwise
except at the endpoints of $I$.  
For each $n$ and for each $k$ we have
\[
\tr (\varphi_n(S)) \leq N(k) \leq \tr (\psi_n(S)).
\]
Now multiply through by $2^{d'/2}\,\left(\frac{\pi}{k}\right)^{d/2}$ and take limits as $k\to\infty$
to get that, for each $n$,
\[
\calF(\varphi_n)\leq \liminf\,2^{d'/2}\,\left(\frac{\pi}{k}\right)^{d/2}N(k)
\leq \limsup\,2^{d'/2}\,\left(\frac{\pi}{k}\right)^{d/2}N(k) \leq \calF(\psi_n)
\]
where we have used the notation \[
\calF(\varphi)  = \int \varphi(s)\, \calD_a(s)\, ds.
\]
By the Lebesgue dominated convergence theorem (recalling that $\calD_a(s)\, ds$
is absolutely continuous)
\[
\lim_{n\to\infty}\calF(\varphi_n) = \calF(\1_I) = \lim_{n\to\infty}\calF(\psi_n),
\]
and the result follows.
\end{proof}

\smallskip
In case $a\equiv 1$
\[
\int_I\calD_1\,ds = \frac{|\Gamma|}{\Gamma(d'/2)}\int_I\, \left(-\log(s)\right)^{\frac{d'}{2}-1}\,
\frac{ds}{s}.
\]
An integration by parts allows one to compute the integral.  The result is as follows:
\begin{corollary}  Assume $\Gamma$ is compact, isotropic or co-isotropic, and $d'>0$.
Let $I = [\lambda, \mu]\subset (0,1]$ and let $N_I(k)$ denote the number of eigenvalues
of $S_{d\sigma}$ in $I$.  Then
\[
\lim_{k\to\infty} 2^{d'/2}\,\left(\frac{\pi}{k}\right)^{d/2}\,N_I(k) = \frac{|\Gamma|}{\Gamma(1+d'/2)}\,
\left[(-\log(\lambda))^{d'/2} -(-\log(\mu))^{d'/2} \right].
\]
\end{corollary}

\subsubsection{Asymptotics of the Schatten norms}
We will also prove the following result on the Schatten norms:
\begin{theorem}\label{Schatten}
Let $\Gamma$ be isotropic or co-isotropic and let $a$ be smooth 
compactly supported complex-valued function on $\Gamma$. Then for every $0<p<\infty,$ 
\begin{equation*}
\lim_{k\rightarrow\infty}2^{d'/2}\left(\frac{k}{\pi}\right)^{-d/2}  \tr\left((S_{ad\sigma}^*S_{ad\sigma})^{p/2}\right)=\int_\Gamma  \frac{\vert a(w)\vert^p}{p^{d'/2}}  d\sigma(w).
\end{equation*}
In particular 
\begin{equation}
\lim_{k\rightarrow\infty}2^{d'/2p}\left(\frac{k}{\pi}\right)^{-d/2p}\Vert S_{ad\sigma}\Vert_p=
\left(\int_\Gamma  \frac{\vert a(w)\vert^p}{p^{d'/2}}  d\sigma(w)\right)^{1/p},
\end{equation}
where $\Vert S_{ad\sigma}\Vert_p$ is the Schatten $p$-norm of $S_{ad\sigma}$.
\end{theorem}

\subsubsection{The limit of the entropy of a mixed state}

As a corollary of the \Sz theorem we can estimate the entropy of the mixed states
mentioned above.
Let us fix $\Gamma\subset\bbR^{2N}$ either an isotropic or co-isotropic submanifold 
with $d'>0$, and $a\in C_0^\infty(\Gamma)$ such that
$a\geq 0$ and $\int_\Gamma a\, d\sigma = 1$.  Then
\[
\rho_a = \left(\frac{\pi}{k}\right)^N \, T_{ad\sigma}
\]
is a mixed state.  Let us denote by $p_1\geq p_2\geq\cdots \geq 0$ the eigenvalues of $\rho_a$ 
listed with multiplicities.  We are interested in the information entropy of $\rho_a$, that is
\[
\bbH(\rho_a) := -\sum_{j=1}^\infty p_j\, \log(p_j).
\]

\medskip
Our \Sz limit theorems are on the spectra of the operators
\[
S_{ad\sigma} = 2^{-d'/2} \left(\frac{\pi}{k}\right)^{N-d/2}\, T_{ad\sigma}
\]
where $d' = 2N-d$ in the co-isotropic case and $d'=d$
in the isotropic case.  In terms of $S_{ad\sigma}$, $\rho_a$ is
\begin{equation}\label{verAqui}
\rho_a = 2^{d'/2}\, \left(\frac{\pi}{k}\right)^{d/2}\, S_{ad\sigma} = C_d\, k^{-d/2}\, S_{ad\sigma},
\end{equation}
where $C_d:= 2^{d'/2}\, \pi^{d/2}.$
If $\mu_j$ are the eigenvalues of $S$ then
$ p_j = C_d\, k^{-d/2}\, \mu_j$, and therefore
\[
\bbH(\rho_a) = -\sum_{j=1}^\infty C_dk^{-d/2}\mu_j\left[\log(C_dk^{-d/2}) + \log(\mu_j)\right].
\]
Taking traces in (\ref{verAqui}) one gets
$C_dk^{-d/2}\, \sum_{j=1}^\infty \mu_j = 1$ and so 
\begin{equation}\label{expresionEntro}
\bbH(\rho_a) = -\log(C_dk^{-d/2}) - C_dk^{-d/2}\,\sum_{j=1}^\infty \mu_j\,\log(\mu_j).
\end{equation}
By the \Sz theorem using the test function $\varphi(s) = s\log(s)$ (which is in our class 
of test functions),
\[
\sum_{j=1}^\infty \mu_j\,\log(\mu_j) = O\left(k^{d/2}\right),
\]
and therefore the second term in (\ref{expresionEntro}) is $O(1)$.  However the first term
is universal (it only depends on $d$).  After a short calculation of constants one can conclude:
\begin{theorem}  If $d'>0$, $a\geq 0$ and $\int_\Gamma a\, d\sigma = 1$
\[
\lim_{k\to\infty} \left[\bbH(\rho_a) +\log(C_dk^{-d/2})\right] = -
\frac{1}{\Gamma(d'/2)}\int_{\Gamma}\left(\int_0^{a(w)}s\log(s)\, 
\log\left(\frac{a(w)}{s}\right)^{\frac{d'}{2}-1}\,
\frac{ds}{s}\right) d\sigma(w).
\]
\end{theorem}
This result has the same form as the general relationship between the
differential entropy of a continuous distribution and the entropy of its discretization in bins of size
$C_dk^{-d/2}$.

\smallskip
Incidentally,
we can directly apply Theorem (\ref{Schatten}) to obtain the limit of the trace distance
between two mixed states associated with the same $\Gamma$; it is just the $L^1$ distance:
\begin{equation}\label{}
\lim_{k\to\infty}\norm{\rho_a - \rho_b}_1 = \int_\Gamma \left| a(w)-b(w)\right|\, d\sigma(w).
\end{equation}

\bigskip
\begin{remark}
We finish this subsection with the following general remark.
In all the previous statements,
the only difference between the isotropic and the co-isotropic cases 
is that, in the latter, $d'$ is the codimension of $\Gamma$ instead of its dimension.
This can be interpreted as follows.
Assume $\Gamma$ is co-isotropic.
Then $\Gamma$ is foliated by leaves tangent to the spaces
$T_w\Gamma^\circ$, $w\in\Gamma$.  The dimension of the leaves is the codimension of $\Gamma$, 
and the leaves are isotropic submanifolds of $\bbR^{2N}$.  It is clear from the definition
that in this case $T_{ad\sigma}$ can be thought of (albeit non rigorously) as an integral of singular 
Toeplitz operators, one for each isotropic leaf.  By Theorem \ref{Szego}
the ``crossed terms" between these isotropic operators do not contribute to the asymptotics
of $\tr[\varphi(S)]$.

The fact that the co-isotropic case formally boils down to a sum of isotropic cases
indicates that the construction of  mixed states, as we are considering here, is more
naturally adapted to the isotropic setting.  This is in agreement with the general 
philosophy that single quantum states aren't naturally associated with co-isotropic submanifods.
Rather, if the null foliation of a co-isotropic is fibrating, $\pi:\Gamma\to X$, then $X$ is
symplectic and to $\Gamma$ one ought to associate the equivalent of the quantizaton
of $X$-worth of quantum states (``quantization commutes with reduction").  

\end{remark}

\subsection{Examples}

We present here some examples.  (For examples with $\bbC^N$ replaced with
a compact K\"ahler manifold see \cite{LF}, where the harmonic oscillator is also
treated.)

\subsubsection{The harmonic oscillator} 
The simplest example is $N=1$ and 
$\Gamma = \{ w\in\calC\;;\; |w| = r\}$ parametrized by $w=re^{it}$ so that $d\sigma = rdt$.  
Straightforward calculations show that, with $\psi(z) = f(z) e^{-k|z|^2/2}\in\calB_k$,
\[
T_\Gamma(\psi)(z) = 
 \frac{k}{\pi} \,r\,e^{-k|z|^2/2}\,e^{-kr^2}\int_0^{2\pi} e^{rkze^{-it}}\, f(re^{it})\, dt.
\]
Clearly $T_\Gamma$ must commute with the $S^1$ representation on $\calB_k$
induced by its action on $\bbC$, so it is diagonal on the monomial
basis consisting of
\[
\ket{n} =  z^n\,e^{-k|z|^2/2},\quad n=0, 1, \ldots.
\]
One computes that
\[
T_\Gamma\ket{n} = \lambda_n\,\ket{n} \quad\text{with}\quad
\lambda_n = 2k\, r\frac{r^{2n}k^n}{n!}\,e^{-kr^2},
\]
from where it follows that
\begin{equation}\label{}
\tr(T_\Gamma) = 2k r,
\end{equation}
the length of the circle times $k/\pi$, in agreement with (\ref{trace}).

The associated mixed sate is $\rho_\Gamma := \frac{1}{2k r}\,  T_\Gamma$.
Its eigenvalues are just
\[
p_n = \frac{r^{2n}k^n}{n!}\,e^{-kr^2}.
\]
With respect to $n$, $p_n$ is exactly a Poisson distribution with $\lambda = kr^2$.

To estimate the operator norm of $T_\Gamma$ we find the maximum eigenvalue $
\lambda_{\text{\tiny max}}$ (the mode of the Possion distribution).
For this one considers the quotients
$
\frac{\lambda_n}{\lambda_{n-1}} = \frac{kr^2}{n}.
$
It follows that $\lambda_{\text{\tiny max}}$ corresponds to $n\cong kr^2$, which, in the 
present context, can be seen as a kind of Bohr-Sommerfeld condition.  
In fact if we impose that $k$ be of the form
$k=r^2/n$, $n=1,2\cdots$,  then 
\[
\lambda_{\text{\tiny max}}=  \lambda_{n=kr^2} = 2kr\,\frac{(kr^2)^{kr^2}}{(kr^2)!}\, e^{-kr^2} 
\sim \sqrt{\frac{2k}{\pi}}
\]
by Stirling's formula, showing that (\ref{estima}) is sharp in this case.  
The associated eigenvector $\ket{n}$ concentrates semi-classically
on $\Gamma$.
The normalized operator is 
$S_\Gamma = \sqrt{\frac{\pi}{2k}} T_\Gamma$
and has greatest eigenvalue $\sim 1$, and the Szeg\"o limit theorem, in this case, 
reads
\begin{equation}\label{}
\lim_{k\to\infty}\sqrt{\frac{2\pi}{k}}\,
 \sum_{n=0}^\infty \varphi\left(\sqrt{2\pi k}\,\frac{r^{2n+1} k^n}{n!} e^{-kr^2}\right) =
2r\sqrt{\pi}\int_0^{1}\varphi(s)\, 
\frac{ds}{s\sqrt{-\log(s)}}
\end{equation}
for all functions $\varphi$ such that $\varphi(s)/s^p$ is continuous on $[0,1]$ for some $p>0$.

\bigskip
We can generalize the previous example to a product of $d$ circles in $\bbC^N$:
\begin{equation}\label{}
\Gamma = \left\{ (r_1e^{it_1},\cdots r_d e^{it_d}, 0\cdots, 0)\in\bbC^N\;;\; 
(t_1,\ldots ,t_d)\in [0,2\pi]^d\right\}.
\end{equation}
Once again the monomials
\[
z^n\, e^{-k|z|^2/2},\qquad n = (n_1,\ldots, n_N)\in\bbN^N
\]
are eigenfunctions of $T_\Gamma$.  The eigenvalues are simply the product of the
one--dimensional eigenvalues, namely
\[
\lambda_n = (2k)^d\, \prod_{j=1}^d r_j\frac{r_j^{2n_j}k^{n_j}}{n_j!}\,e^{-kr_j^2} =
\frac{(2k)^d}{n!}\, k^{|n|}\, r^{2|n|+d}\, e^{-k|r|^2},
\]
where $|r|^2 = \sum r_j^2$, $n! = \prod n_j!$, $|n| = \sum n_j$.
Once again estimating the greatest eigenvalue shows that the operator norm of
$T_\Gamma$ is $O(k^{N-d/2})$, showing that in general (\ref{estima}) is sharp.

\subsubsection{A symplectic but non-complex example}
This example shows that $\Delta_n(w)$ does not have to be constant with respect
to $w\in\Gamma$. Let $z=(z_1, z_2)$ be the variable in $\bbC^2$, and let
$z_j = x_j + \sqrt{-1} y_j$.  Let $\Gamma$ be defined by the equations:
\[
\Gamma:\qquad x_2 = \frac 12 x_1^2\quad\text{and}\quad y_2 = 0.
\]
Then $(x_1, y_1)\mapsto \left(x_1, y_1, \frac{1}{2}x_1^2, 0\right)$ is a parametrization
of $\Gamma$, and $\left\{\langle 1, 0, x_1, 0\rangle ,\, \langle 0, 1, 0, 0\rangle\right\}$
is the moving frame associated to it.  Clearly $\Gamma$ is a symplectic submanifold.
An elementary calculation shows that the matrix of $K$ in the parametrization is
\[
\begin{pmatrix}
0 & -\frac{1}{1+x^2}\\
1 & 0
\end{pmatrix},
\]
and therefore $r=1$, $\lambda_1 = (1+x_1^2)^{-1/2}$ and
\[
 \Delta_n = \sum_{j=0}^{\floor{\frac{n-1}{2}}} {n\choose 2j+1}\, \frac{1}{(1+x_1^2)^j}
 \]  
by the identity (\ref{alternatively}).

\bigskip
The paper is organized as follows.  In the next section we compute the covariant symbol of
$T_{ad\sigma}$ and its kernel, and derive some easy consequences.  Theorem \ref{NormEst}
is proved in \S 3.  In \S 4 we establish Theorem \ref{ManyOps}
and the Szeg\"o theorem (Theorem \ref{Szego}), first for polynomials and then for
general test functions $\varphi$.  A key step in this extension is to show that
 for every $p\in (0,1)$
$\tr(S^p_{ad\sigma})$ is  $O(k^{d/2})$, which we do in \S 4.3.
In \S 5 we prove the Theorem on the Schatten norms.
Finally, in \S 6 we consider the case when $\Gamma$ is a Lagrangian submanifold
satisfying the Bohr-Sommerfeld condition and obtain lower bounds for the maximum
eigenvalue of $T_{ad\sigma}$, by using as test functions Lagrangian pure states associated with
$\Gamma$ (see Proposition \ref{LowerBound}).  

\smallskip
Our proofs are direct, based on the explicit formula for the reproducing kernel and the 
method of stationary phase.  It is expected however 
that there is a symbol calculus for a class of operators that includes
the operators treated here, and generalizations.  (For example, one can envision 
forming mixed states by integrating projectors over more general coherent states.)
The symbols will be symplectic spinors, as in \cite{BG} and \cite{GUW}.  Such a symbol calculus
will allow us to deal with many other issues related to these operators, for example propagation 
under a quantum Hamiltonian, as well as the extension of the theory to quantized compact
K\"ahler manifolds.  Since the \Sz projector on such manifolds has the same form asymptotically
as in the Euclidean case, it is clear that the results presented here will take the same 
general form in that setting.

\section{Symbols and kernels}

We keep the notation of the previous section.

\subsection{The covariant symbol}

By definition, the covariant (Wick or Berezin) symbol
of $T_{ad\sigma}$ is the function
\[
\widetilde{a\sigma}_k= \frac{\inner{T_{ad\sigma}(e_z)}{e_z}}{\inner{e_z}{e_z}} = 
\frac{T_{ad\sigma}(e_z)(z)}{\Pi(z,\zbar)}.
\]
A straightforward calculation shows:
\begin{equation}\label{}
\widetilde{a\sigma}_k = \constants\, \int_\Gamma e^{-k|z-w|^2}\, a(w)\, d\sigma (w).
\end{equation}
Note that in particular $\widetilde{a\sigma}_k$ is exponentially small as $k\to\infty$ unless $z\in \Gamma$.
Moreover, the knowledge of $\widetilde{a\sigma}_k$ for every $k$ determines $ad\sigma$. 
In fact the previous expression shows that   $\widetilde{a\sigma}_k$ is the heat evolution  at $t=1/4k$ of the  measure $ad\sigma$ in $\C ^{N}$, and therefore
 \begin{equation*}
  \lim_{k\rightarrow\infty}\widetilde{ad\sigma}_k\,dv=ad\sigma
  \end{equation*} 
  in the $w^{*}$ topology of $C_0(\C ^N)^{*}$.

Also, by the general theory of Berezin (\cite{Be})
\begin{equation}\label{trazaCovariante}
\tr(T_{ad\sigma}) = \constants\int_{\bbC^N}\, \widetilde{a\sigma}_k\, dL(z),
\end{equation}
from which one can easily recover (\ref{trace}).

\medskip
Another general formula due to Berezin is that the trace of the composition of
two operators is the integral of the product of the covariant symbol of one times the
Toeplitz (or anti-Wick or contravariant) symbol of the other.  This leads to:

\begin{lemma}
Let $a,\, b\in C_0^\infty(\Gamma)$.  Then
\begin{equation}\label{trazaDeDosExacta}
\tr\left(T_{ad\sigma}\circ T_{bd\sigma}\right) = \left(\frac{k}{\pi}\right)^{2N}
\iint_{\Gamma\times\Gamma}e^{-k|z-w|^2}\,
a(w)\, b(z)\, d\sigma (w)\,d\sigma(z).
\end{equation}
\end{lemma}
Below we generalize (\ref{trazaDeDosExacta}) to a composition of $n$ operators.

\medskip
It is straightforward to estimate the trace of an ordinary Berezin-Toeplitz operator
composed with $T_{ad\sigma}$.  Let $H:\bbC^N\to\bbC$ be smooth and, say, of polynomial
growth as well as all its derivatives.  Define
\[
\forall \psi\in \calB_k\qquad T_H(\psi) = \Pi_k(H\psi).
\]
Then we have
\[
\tr (T_H\circ T_{ad\sigma}) = \left(\frac{k}{\pi}\right)^{2N}
\int_{\bbC^N}\int_\Gamma e^{-k|z-w|^2}
H(z)\, a(w)\, d\sigma(w) \, dL(z),
\]
which, by the method of stationary phase, implies the following:
\begin{proposition}  As $k\to\infty$
\[
 \tr (T_H\circ T_{ad\sigma}) =  \constants \int_\Gamma H(w)\, a(w)\, d\sigma(w) + 
 O(k^{N-1})
\]
(in fact there is a full asymptotic expansion of this trace).
\end{proposition}

\subsection{The Wick kernel}
Let $A:\calB_k\to\calB_k$ be any bounded
operator.  By the reproducing property of the coherent
states
$
A(\psi) = \int_{\bbC^N}\psi(w)\,A(e_w)\, dw,
$
or
\begin{equation}\label{}
A(\psi)(z) = \int_{\bbC^N} \psi(w)\,\inner{A(e_w)}{e_z}\, dw.
\end{equation}
This shows that
\begin{equation}\label{sKernel}
\calK_A (z,w) := \inner{A(e_w)}{e_z}
\end{equation}
acts as the Schwartz kernel for $A:\calB\to\calB$. We will refer to $\calK_A$  as the
(Wick or Berezin) kernel of $A$.  It is easy to check that if $B:\calB_k\to\calB_k$
is another operator, then
\begin{equation}\label{compoDos}
\calK_{A\circ B}(z,w) = \int_{\bbC^N} \calK_A(z,\zeta)\,\calK_B(\zeta,w)\, d\zeta.
\end{equation}
Furthermore, (\ref{trazaCovariante}) is equivalent to
\begin{equation}\label{trazaKernel}
\tr(A) = \int_{\bbC^N}\, \calK_A(z,z)\, dL(z).
\end{equation}

The following is immediate:
\begin{lemma}\label{TheKernel}
Let $\Gamma\subset\bbC^n$ be a  submanifold and $d\sigma$ a fixed positive 
measure on it.  For  each $a\in C_0^\infty(\Gamma)$,
the Berezin kernel (\ref{sKernel}) of $T_{ad\sigma}$ is
\begin{equation}\label{}
\calK_a(z,w) = \int_\Gamma \Pi(z,\bar\zeta)\,\Pi(\zeta, \wbar)\, a(\zeta)\, d\sigma(\zeta) =
\end{equation}
\[
= \left(\frac{k}{\pi}\right)^{2N}\int_\Gamma e^{ik\left[ \omega(\zeta, w-z) + 
i\left( |z-\zeta|^2 + |\zeta-w|^2\right)/2\right]} a(\zeta)\, d\sigma(\zeta).
\]
\end{lemma}

For future use we record here a formula for the trace of a composition of $n$
singular Toeplitz operators associated with the same $\Gamma\subset\bbC^N$:

\begin{proposition} For $n\geq 2$ and smooth functions $a_j\in C_0^\infty(\Gamma)$,
$j=1,2,\ldots n$,
\begin{equation}\label{trazaDeMuchos}
\tr(T_{a_1d\sigma}\circ\cdots \circ T_{a_nd\sigma})=\left(\frac{k}{\pi}\right)^{Nn} \int_{\Gamma^n}\left(e^{ik\Phi (\zeta_1\cdots,\zeta_n) }\prod_{i=1}^n a_i(\zeta_i)\right) d\sigma(\zeta_1)\cdots \sigma(\zeta_n),
\end{equation}
where 
\begin{equation}\label{}
\Phi(\zeta_1,\cdots,\zeta_n)=\frac{i}{2}\left(\vert \zeta_1-\zeta_2\vert^2+\cdots+\vert
\zeta_{n}-\zeta_1\vert^2\right)
+ \omega(\zeta_1,\zeta_{2})+\cdots +\omega(\zeta_{n},\zeta_{1})
\end{equation}
and $\omega$ is the symplectic form $\omega(z,w)=\frac{1}{2i}(z\barw -w\overline{z})$.
\end{proposition}
\begin{proof}
Let us denote
by $\calK$ the Wick
kernel of the composition $T_{a_1d\sigma}\circ\cdots \circ T_{a_nd\sigma}$.
Then by induction on (\ref{compoDos}) one can prove that
\begin{equation}\label{compoN}
\calK(z,\barw)=\int_{\Gamma^n}\Pi (z,\overline{\zeta_1})
\prod_{i=1}^{n-1}\Pi (\zeta_i,\overline{\zeta_{i+1}})\, \Pi (\zeta_n,\overline{w})\, 
a_1(\zeta_1)\, d\sigma(\zeta_1)\cdots a_n(\zeta_n)\,d\sigma(\zeta_n).
\end{equation}
The desired trace, (\ref{trazaDeMuchos}), is the integral $\int_{\bbC^N}\calK(z,\zbar)\, dL(z)$.
Using the reproducing property 
\[
\int_{\bbC^N} \Pi(z,\overline\zeta)\,\Pi(\zeta,\zbar)\,dL(z) =
\Pi(\zeta,\overline\zeta)
\]
one obtains
\begin{equation}\label{}
\int_{\bbC^N}\calK(z,\zbar)\, dL(z) = 
\int_{\Gamma^n}\prod_{i \, mod\,n}\Pi(\zeta_i,\overline{\zeta_{i+1}})
\prod_{1=1}^{n}a_i(\zeta_i)\,d\sigma(\zeta_1)\cdots d\sigma(\zeta_n), 
\end{equation}
and the result follows.
\end{proof}

\section{Proof of the norm estimate}
\renewcommand{\S}{\mathbb{S}}
We now prove Theorem \ref{NormEst}.  We need the following:
 \begin{lemma}
 Let $f$ be holomorphic in a neighborhood of $B(a,r)\subset\C^N$. 
 Then
 \begin{equation*}
 \vert f(a)\vert^2e^{-k\vert a\vert^2}\leq \frac{k^N}{\gamma(kr^2)}\int_{B(a,r)}\vert f(z)\vert^2 e^{-k\vert z\vert^2}dv(z),
 \end{equation*}
 where $\gamma(s)=\frac{\Theta_N}{2}\int_0^{s}t^{N-1}e^{-t}dt$ and $\Theta_N$ is the surface area of the unit sphere in $\C^N$. 
 \end{lemma}
 \begin{proof}
 Let $\S$ be the unit sphere in $\C^N$, $\tau$ the Lebesgue measure in $\S$ so that 
 $\Theta_N=\tau(\S)$.
Then the sub-harmonicty of the function of $w$, 
 \begin{equation*}
 \vert f(a+\rho w) e^{\rho w\cdot \overline{a}}\vert^2,
 \end{equation*}
 for $a$ and $\rho$ fixed implies that
 \begin{align*}
 \int_{B(a,r)}\vert f(z)\vert^2 e^{-k\vert z\vert^2}dv(z)=&\int_{B(0,r)}\vert f(a+\zeta) 
 e^{k\zeta\cdot \overline{a}}\vert^2 e^{-k\vert \zeta\vert^2}dv(\zeta)\, e^{-k\vert a\vert^2}\\
 =&\int_0^r\int_{\S}\vert f(a+\rho\zeta') e^{k\rho\zeta'\cdot \overline{a}}\vert^2 d\tau(\zeta')\ e^{-k\rho^2}\rho^{2N-1}\, e^{-k\vert a\vert^2}\,d\rho\\
 \geq &\Theta_N\, \vert f(a)\vert^2\,e^{-k\vert a\vert^2}\int_0^r \rho^{2N-1}e^{-k\rho^2}d\rho\\
  =& \frac{\Theta_N}{2k^N}\,\vert f(a)\vert^2\,e^{-k\vert a\vert^2}\int_0^{kr^2}t^{N-1}e^{-t}dt.
 \end{align*}
 \end{proof}
 \begin{proof} (of Theorem (\ref{NormEst}))
Assume without loss of generality that $a\geq 0$. Let $r>0$ and  and  $\lbrace z_n\rbrace$ a numbering of the lattice $r\mathbb{Z}^{2N}$ in $\C^N$. For every $\lambda>0$, there exists a number $L\in\N$ such that no more than $L$ balls $B(z_n, \lambda r)$ intersect for any $n$. The constant $L$ is independent of $r$.  
Let $d\mu=ad\sigma$. Notice that since $\Gamma$ is smooth and $a\in L^{\infty}$  then $$\mu(B(z,r))\leq M\Vert a\Vert_\infty r^d.$$
Since $\C\subset\cup_n B(z_n,2Nr)$ 
(the diameter of a cube in $\R^{2N}$ of side $r$ is $\sqrt{2N}r$),
then if we let $L$ correspond to $\beta=(2N+1)r$ in the argument above we have,
for $\psi=f(z)e^{-k\vert z\vert^2}\in \mathcal{B}_k$,
\begin{align*}
\langle T_{ad\sigma}\psi,\psi\rangle&=
\int_{\C}\vert f(z)\vert^2 e^{-k\vert z\vert^2}d\mu(z)\\
&\leq \sum_n \int_{B(z_n,2Nr)}\vert f(z)\vert^2 e^{-k\vert z\vert^2}d\mu(z)\\
&\leq\frac{k^N}{\gamma(kr^2)}\, \sum_n \int_{B(z_n,2Nr)} \left(
\int_{B(z,r)}\vert f(u)\vert^2 e^{-k\vert u\vert^2}dv(u)\right)d\mu(z)\\
&\leq \frac{k^N}{\gamma(kr^2)}\, \sum_n \int_{B(z_n,2Nr)} \int_{B(z_n, (2N+1)r)}\vert f(u)\vert^2 e^{-k\vert u\vert^2}dv(u)\, d\mu(z)\\
&=\frac{k^N}{\gamma(kr^2)}\, \sum_n \mu(B(z_n,2Nr))\int_{B(z_n, (2N+1)r)}\vert f(u)\vert^2 e^{-k\vert u\vert^2}dv(u)\\
&\leq \frac{k^NM(2Nr)^d\Vert a\Vert_\infty}{\gamma(kr^2)}\,\sum_n \int_{B(z_n, (2N+1)r)}\vert f(u)\vert^2 e^{-k\vert u\vert^2}dv(u)\\
&\leq \frac{k^NLM(2Nr)^d\Vert a\Vert_\infty}{\gamma(kr^2)} \int_{\C}\vert f(u)\vert^2 e^{-k\vert u\vert^2}dv(u).
\end{align*}
If we choose $kr^2=1$ we obtain that for a constant $C>0$
\begin{equation*}
\int_{\C}\vert f(z)\vert^2 e^{-k\vert z\vert^2}d\mu(z)\leq C\Vert a\Vert_\infty  k^{N-d/2}\int_{\C}\vert f(u)\vert^2 e^{-k\vert u\vert^2}dv(u), 
\end{equation*}
hence
\begin{equation*}
\langle T_{ad\sigma}\psi,\psi\rangle\leq C\Vert a\Vert_\infty\,  k^{N-d/2}\,\Vert \psi\Vert^2_{\calB_k} 
\end{equation*}for every $\psi\in \calB_k$ and the proposition follows.
\end{proof}

\section{Proof of the \Sz limit theorem}

The proof of Theorem \ref{Szego} begins with the proof of Theorem \ref{ManyOps}.
After that we show that the traces of the $S$ operators are bounded, and a final argument
concludes the proof.

\subsection{Proof of Theorem \ref{ManyOps}}

 The starting point of the proof
 is the expression (\ref{trazaDeMuchos}) for the trace of $\Upsilon$,
 which we recall for convenience:
 \[
 \tag{\ref{trazaDeMuchos}}
\tr(T_{a_1d\sigma}\cdots T_{a_nd\sigma})=\left(\frac{k}{\pi}\right)^{Nn} 
\int_{\Gamma^n}\left(e^{ik\Phi (\zeta_1\cdots,\zeta_n) }\prod_{i=1}^n a_i(\zeta_i)\right) 
d\sigma(\zeta_1)\cdots \sigma(\zeta_n),
\]
where 
\[
\Phi(\zeta_1,\cdots,\zeta_n)=\frac{i}{2}\left(\vert \zeta_1-\zeta_2\vert^2+\cdots+\vert
\zeta_{n}-\zeta_1\vert^2\right)
+ \omega(\zeta_1,\zeta_{2})+\cdots +\omega(\zeta_{n},\zeta_{1}).
\]
We will estimate this trace using the method of stationary
 phase.  
 
 \medskip
Note first that the 
imaginary part of $\Phi$ is non-negative and is zero precisely on the diagonal
\[
\Gamma^\Delta := \{(\zeta_1,\ldots , \zeta_n)\in\Gamma^n\;;\; 
\zeta_1= \zeta_2 = \cdots = \zeta_n\}.
\]
For this reason one can easily show that the integrand of (\ref{trazaDeMuchos}),
integrated over the complement of any neighborhood of $\Gamma^\Delta$,
is exponentially decreasing.  Therefore, asymptotically to all polynomial orders we can
restrict our attention to a small neighborhood of $\Gamma^\Delta$.

In addition we will  show that
\[
\tag{$\ast$} \Gamma^\Delta\ \text{is a non-degenerate manifold of critical points of } \Phi.
\]
In particular, in a sufficiently small neighborhood of $\Gamma^\Delta$ any
critical point of $\Phi$ is in $\Gamma^\Delta$.  

 \bigskip
 \newcommand{\g}{\vec{\gamma}}
 \smallskip
 Let $\gamma : B(0;1)\subset \R ^d\rightarrow \Gamma$ be a parametrization of an 
 open set $\calU\subset\Gamma$.  
 The notation
 \begin{equation}\label{}
\g_j :=\frac{\partial\gamma}{\partial t_j}\in\bbR^{2N}
\end{equation}
will be useful in the proof.
 Let  $G(t)=(g_{ij}(t))$ be the metric of $\Gamma$, that is 
 $ g_{ij}(t)=\g_i (t)\cdot \g_{j}(t)$, and let $d\mu(t)dt=\sqrt{\det(G)}\,dt$ denote
 the volume element on $\Gamma$.
 
 Let $t,\, s_1,\, \cdots \,s_{n-1}$ be variables in $\bbR^d$
 so that, for example $ s_i = (u^i_1,\ldots , u^i_d)\in\bbR^d$,
 and define $\Xi: B(0,\epsilon)^n\to (\bbR^{2N})^n$ by
 \begin{equation}\label{}
 \Xi(t,\, s_1,\, \cdots \,s_{n-1}) = 
 (\gamma(t),\gamma(t+s_1),\cdots ,\gamma(t+s_{n-1})).
\end{equation}
Choosing $\epsilon$ small enough, this 
is a parametrization of an open set $\calV\subset\Gamma^n$ intersecting the diagonal
in $\calU^\Delta$.  This intersection corresponds to $s_1 = \cdots = s_{n-1} = 0$.
 
 \renewcommand{\s}{\mathbf{s}}
 \smallskip
Let $\chi\in C_0^\infty(\calV)$.  We begin by calculating the asymptotic
expansion as $k\to\infty$ of
\begin{equation}\label{carta} 
 I_\chi(k)=\left(\frac{k}{\pi}\right)^{Nn}\int_{\calV} e^{ik\Phi(\zeta_1,\cdots,\zeta_n)} 
\ \prod_{j=1}^n a_j(\zeta_j)\ \chi(\zeta_1,\cdots,\zeta_n)\ d\sigma(\zeta_1)\cdots d\sigma(\zeta_n),
\end{equation}
computing in the parametrization $\Xi$.  We will 
do stationary phase with respect to the $\s = (s_1,\ldots, s_{n-1})$ 
variables for each $t$, and then integrate over $t$.

Let
\begin{equation*}
\psi(t,\s)= \Phi\circ\Xi(t,\s) = \Phi(\gamma(t), \gamma(t+s_1),\gamma(t+s_2),\cdots, \gamma(t+s_{n-1}))
 \end{equation*} 
be the phase in coordinates, and write
 \begin{equation*}
 \psi(t,s)=i\psi_1(t,s) +\psi_2(t,s)
 \end{equation*}
where
\begin{equation}\label{}
\psi_1(t,\mathbf{s})=\frac{1}{2}\left| \gamma(t)-\gamma(t+s_1)\right|^2
+\frac{1}{2}\sum_{i=1}^{n-2}\left|\gamma(t+s_i)-\gamma(t+s_{i+1})
\right|^{2}+\frac{1}{2}\left| \gamma(t+s_{n-1})-\gamma(t)\right|^2 
\end{equation}
and
\begin{equation}
\psi_2(t,\mathbf{s})=\omega(\gamma (t),\gamma(t+s_1))
+ \sum_{i=1}^{n-2}\omega (t+s_i ,t+s_{i+1}) +\omega(\gamma(t+s_{n-1}),\gamma(t)).
\end{equation}

We claim that for $t$ fixed  $\psi$ has a critical point at $\mathbf{s}=0$. 
In fact, let $s_i=(u^i_1,\cdots,u^i_d)$, $i=1,\cdots, n-1$.  Then for $i=2,\cdots ,n-2$ and $j=1,\cdots, d$
\begin{align}\label{grad1}
\nonumber\frac{\partial\psi_1(t,\mathbf{s})}{\partial u^i_j}=&(\gamma(t+s_i)-\gamma(t+s_{i+1}))\cdot\g_j(t+s_i)\\
 +&(\gamma(t+s_{i})-\gamma(t+s_{i-1}))\cdot\g_j(t+s_{i-1}),
\end{align}
and
 \begin{align}\label{grad2}
\nonumber\frac{\partial\psi_1(t,\mathbf{s})}{\partial u^1_j}=&(\gamma(t+s_1)-
\gamma(t+s_{2}))\cdot\g_j(t+s_1)\\
 +&(\gamma(t+s_{1})-\gamma(t))\cdot\g_j(t+s_1).
\end{align}
In addition
  \begin{align}\label{grad3}
\nonumber\frac{\partial\psi_1(t,\mathbf{s})}{\partial u^{n-1}_j}=&(\gamma(t+s_{n-1})-\gamma(t+s_{n-2})\cdot\g_j(t+s_{n-1}))\\
 +&(\gamma(t+s_{n-1})-\gamma(t))\cdot\g_j(t+s_{n-1}).
\end{align}

For $\psi_2 $ we have as before for $i\neq 1, n-1$
\begin{equation}\label{grad4}
\frac{\partial\psi_2}{\partial u_j^{i}}(t,\mathbf{s})=
\omega(\g_j(t+s_i),\gamma(t+s_{i+1}))+\omega(\gamma(t+s_{i-1}),\g_j(t+s_i)),
\end{equation} and for $i=1,n-1$
\begin{equation}\label{grad5}
\frac{\partial\psi_2}{\partial u_j^{i}}(t,\mathbf{s})=
\omega(\g_j(t+s_i),\gamma(t))+\omega(\gamma(t),\g_j(t+s_i)).
\end{equation}
It follows that $\nabla_{\mathbf{s}}\psi(t,0)=0$.

\medskip
We can calculate directly from \eqref{grad1},\eqref{grad2},\eqref{grad3}   
the Hessian matrix $\psi_{1}(t,\cdot)''(0)$ with respect to the $\s$ variables
at $\s=0$. It is the $(n-1)d\times (n-1)d$ matrix

\begin{equation}\label{hess1}
L_{n-1}^{(1)}=\begin{pmatrix}
2G(t) & -G(t) & 0& &\dots&0\\
- G(t)& 2G(t) & -G(t)& &\dots&0\\
0&- G&2G(t) & -G(t)&\dots &0\\
                                        \\
         0& &\dots &\ddots&\ddots&\ddots\\                               
0&& \dots& 0&-G(t)&2G(t)

\end{pmatrix}.
\end{equation}
written in $d\times d$ blocks.

\smallskip

We now compute the Hessian of $\psi_2$, using (\ref{grad4}) and
(\ref{grad5}).  Obviously the partial derivative
$\frac{\partial^2\psi_2}{\partial u_{j'}^{i'}\partial u_j^{i}}(t,0)$ is zero
if the indices $i$, $i'$ are not consecutive modulo $n$ (assigning to the $t$ variables the
index zero).
If the indices are consecutive mod $n$, one can check that
\begin{equation}
\frac{\partial^2\psi_2}{\partial u_{j'}^{i+1}\partial u_j^{i}}(t,0)=
\omega(\g_{j'}  (t),
\g_{j} (t))
\end{equation}
because terms containing second derivatives of $\gamma$ appear in pairs that cancel
each other out, by the skew symmetry of $\omega$.  It follows that the Hessian of $\psi_2$ is

\begin{equation}\label{hess2}
L_{n-1}^{(2)}=\begin{pmatrix}
0 & H(t) & 0& &\dots&0\\
- H(t)& 0 & H(t)& &\dots&0\\
0&- H(t)&0 & H(t)&\dots &0\\
                                        \\
         0& &\dots &\ddots&\ddots&\ddots\\                               
0&& \dots& 0&-H(t)&0
\end{pmatrix}
\end{equation}
where $H$ is the $d\times d$ skew-symmetric matrix
\begin{equation}\label{matrizH}
H(t) = \begin{pmatrix}
\omega(\g_i(t,0), \g_j(t, 0))
\end{pmatrix}.
\end{equation}
In conclusion, the Hessian of the full phase $\psi = i\psi_1 + \psi_2$ with
respect to the $\s$ variables is
the block tri-diagonal Toeplitz matrix 
\[
\text{Hess}_{n-1}(t) = iS_{n-1}(t)
\]
where
\begin{equation}\label{theHessian}
S_{n-1}(t)= 
\begin{pmatrix}
2 G(t) & - G(t)-iH(t) & 0& \dots&0\\
- G(t)+i H(t)& 2 G(t) & - G(t)-iH(t)& \dots&0\\
0&- G(t)+i H(t)&2 G(t) & \dots &0\\
                                       \\
         0& &\dots &\ddots\ddots& - G(t)-iH(t)\\                               
0&& \dots& - G(t)+iH(t)&2 G(t)
\end{pmatrix}.
\end{equation}
It is shown in the appendix (see (\ref{laRaizCuadrada})) that
\begin{equation}\label{inAppendix}
\det\left(\frac 1i\text{Hess}_{n-1}\right) =\det(G)^{n-1}\, \Delta_n^2.
\end{equation}
In particular the determinant of this matrix is positive.  Since the $\s$ variables are variables normal to 
$\Gamma^\Delta$, this proves the claim ($\ast$).

\smallskip

Choose $\epsilon>0$ so that $\s=0$ is the only critical point of the phase with respect 
to the $\s$ variables and let $h_\chi(t,\s) = \chi\circ\Xi (t, \s)$.  Then, by 
 the stationary phase method, 
\begin{equation*}
\int_{B(0,\epsilon)^{n-1}} e^{ik\psi(t,\mathbf{s})}\ a_1(\gamma(t))\, a_2(\gamma(t+s_1))\cdots 
a_n(\gamma(t+s_{n-1}))\, h_\chi(t,\mathbf{s}\,d\mu(t+s_1)\cdots d\mu(t+s_{n-1}))\,
d\s
\end{equation*}
\begin{equation*}
=\left(\frac{2\pi}{k}\right)^{d(n-1)/2}[\det(-i\text{Hess}_{n-1})(t)]^{-1/2}\ \left(\prod_{j=1}^n a_j(\gamma(t))\right)
\, h_\chi(t,0)\, d\mu(t)^{n-1}\left(1+O(1/k)\right)
\end{equation*}
\begin{equation*}
 =\left(\frac{2\pi}{k}\right)^{d(n-1)/2}\Delta_n^{-1}\  \left(\prod_{j=1}^n a_j(\gamma(t))\right)
 \,h_\chi(t,0)\left(1+O(1/k)\right).
\end{equation*}
Notice that the factor of $\sqrt{\det(G)^{n-1}}$ in the square root of (\ref{inAppendix}) 
cancels the Jacobian factors $d\mu(t)$.

Integrating the last expression on $B(0,\epsilon)$ with respect to $d\mu (t)dt$ we obtain
\begin{equation}\label{asintcarta} 
I_\chi(k)=2^{d\frac{n-1}{2}}\left(\frac{k}{\pi}\right)^{n(N-d/2)+d/2}\Delta_n^{-1}
\int_{\Gamma}\prod_{j=1}^n a_j(\zeta)\, \chi^\Delta(\zeta)\,d\sigma(\zeta)\, \left(1+O(1/k)\right)
\end{equation}
where $\chi^\Delta(\zeta) = \chi(\zeta,\,\zeta,\ldots , \zeta).$

\medskip
We now let $\{\chi_\alpha\}$ denote a partition of unit of a neighborhood of 
$\Gamma^\Delta$ in $\Gamma^n$ so that $\sum_\alpha \chi_\alpha\equiv 1$ in a neighborhood 
of the support of $\prod_{j=1}^n a_j(\zeta_j)$,  subordinated to a cover for which the previous
calculations apply.  Then
\begin{equation}\label{}
\tr(\Upsilon) = \sum_\alpha I_{\chi_\alpha}(k) + O(k^{-\infty}),
\end{equation}
and using \eqref{asintcarta}  we obtain
\[
\tr(\Upsilon)=2^{d\frac{n-1}{2}}\left(\frac{k}{\pi}\right)^{n(N-d/2)+d/2}\,
\sum_{\alpha} \int_{\Gamma} \Delta_n^{-1}\left(\prod_{j=1}^n a_j(\zeta)\right)\,
\chi^\Delta_{\alpha}(\zeta)d\sigma (\zeta)\left(1+O(1/k)\right).
\]
\begin{equation}
=2^{d\frac{n-1}{2}}\left(\frac{k}{\pi}\right)^{n(N-d/2)+d/2}\,\int_{\Gamma}
\frac{\prod_{j=1}^n a_j(\zeta)}{\Delta_n}
\,d\sigma (\zeta)\left(1+O(1/k)\right)
\end{equation}
since $\sum_\alpha\chi^\Delta = 1$ in the support of $\prod_j a_j$.
\hfill $\square$.

\subsection{The \Sz limit theorem for polynomials}

We now specialize to the case when $\Gamma$ is isotropic or co-isotropic.
\begin{proposition}
Let $a_j\in C_0^\infty(\Gamma)$, $j=1,\ldots, n$, and assume $\Gamma$
is isotropic or co-isotropic.  Then
\begin{equation}\label{muchosS}
2^{d'/2}\left(\frac{\pi}{k}\right)^{d/2}\tr(S_{a_1d\sigma}\cdots S_{a_nd\sigma})=\frac{1}{n^{d'/2}}\int_\Gamma\prod_{j=1}^n a_j(w)\,d\sigma(w)+O(1/k).
\end{equation}
\end{proposition}
\begin{proof}
First assume that $\Gamma$ is isotropic.  Then $d'=d$ and
\[
S_{a_jd\sigma} = 2^{-d/2}\,\left(\frac{\pi}{k}\right)^{N-d/2}\, T_{a_jd\sigma}, 
\] 
and so by (\ref{laMeraPapa})
\[
\tr(S_{a_1d\sigma}\cdots S_{a_nd\sigma}) = \left(\frac{k}{2\pi}\right)^{d/2}
\left(\int_{\Gamma}\frac{\prod_{j=1}^n a_j(w)}{\Delta_n}\, d\sigma + O(1/k)\right).
\]
Moreover, by (\ref{usefulFact})  the endomorphism $K$ is 
identically zero, and therefore
\begin{equation}\label{}
\forall n=1,2,\cdots\qquad \Delta_n = n^{d/2}.
\end{equation}

\medskip
Let us now assume that $\Gamma$ is co-isotropic, so that $d'=2N-d$ and
\[
S_{a_jd\sigma} = 2^{-N+d/2}\,\left(\frac{\pi}{k}\right)^{N-d/2}\, T_{a_jd\sigma}.
\]
Once again by (\ref{laMeraPapa}) we can conclude that
\begin{equation}\label{}
\tr(S_{a_1d\sigma}\cdots S_{a_nd\sigma}) = 2^{n(d-N)}\, \left(\frac{k}{2\pi}\right)^{d/2}
\left(\int_{\Gamma}\frac{\prod_{j=1}^n a_j(w)}{\Delta_n}\, d\sigma + O(1/k)\right).
\end{equation}
We now compute $\Delta_n$.  By (\ref{usefulFact}),  at each point
in $\Gamma$
\[
\ker(K) = T_w\Gamma^\circ.
\]
Recall that $r$ is half the rank of $K$, and since $ T_w\Gamma^\circ$ and $ T_w\Gamma$
have complementary dimension the dimension of $\Gamma$ must equal
\[
d = N+r.
\]
In particular $r$ is constant.  Note that $d' = N-r$.

To continue we  need a lemma from linear algebra:
\begin{lemma}
If $V$ is co-isotropic, the image $E$ of $K$ is the maximal complex subspace of $V$:
\begin{equation}\label{}
E = J(V)\cap V.
\end{equation}
\end{lemma}
\begin{proof}
The inclusion $\supset$ is obvious.  To prove the reverse inclusion, 
let $v = \Pi(J(u))$ with $u\in V$.  Then
$v = J(u)+w$ with $w\in V^\bot = J(V^\circ)$.  Therefore $\exists a\in V^\circ$ such that
$w=J(a)$ and finally $v=J(u+a)$ with $u+a\in V$.
\end{proof}
Therefore $V =  E\oplus V^\circ$
and the mapping $K:V\to V$ is diagonal with respect to this decomposition.  It is zero
on $V^\circ$ and agrees with $J$ on E.  Therefore, there exists a basis of $V$ where
the matrix for the transformation $K$ is the canonical one, namely 
\begin{equation}\label{}
W_{\text{can}} =\begin{pmatrix}
0 & -I_r & 0\\
I_r & 0 & 0\\
0 & 0 & 0
\end{pmatrix}.
\end{equation}
In particular all the eigenvalues $\lambda_\ell$ are equal to one, and one computes
\begin{equation}\label{}
\forall n=1,2,\ldots \qquad \Delta_n = n^{\frac{N-r}{2}}\, 2^{r(n-1)} = n^{d'/2}\,2^{(n-1)(d-N)}.
\end{equation}
It follows that 
$
\frac{2^{n(d-N)}}{\Delta_n} = \frac{2^{d-N}}{n^{d'/2}}
$
and
\begin{equation}\label{}
\tr(S_{a_1d\sigma}\cdots S_{a_nd\sigma}) = 2^{d-N}\, \left(\frac{k}{2\pi}\right)^{d/2}
\left(\int_{\Gamma}\frac{\prod_{j=1}^n a_j(w)}{n^{d'/2}}\, d\sigma + O(1/k)\right).
\end{equation}
As the constant in front of the integral is
$
2^{d/2-N} \left(\frac{k}{\pi}\right)^{d/2} = 2^{-d'/2} \left(\frac{k}{\pi}\right)^{d/2},
$
the proposition is proved.
\end{proof}

Taking $a_1=a_2=\cdots = a_n = a$, using the linearity of the trace
and since $\calO_{-n/2}(s^n)(t) = \frac{t^n}{n^{d/2}}$, we immediately obtain:
\begin{corollary}\label{TracePowers}  
The \Sz limit Theorem \ref{Szego} holds for $\varphi$ a polynomial
without constant term.
\end{corollary}

\subsection{Bounding the traces}
Let us now assume that $a\geq 0$ and that $\Gamma$ is isotropic or co-isotropic.
The next step in order to obtain Theorem \ref{Szego} is to
show that the rescaled traces $k^{-d/2}\tr\left( S_{ad\sigma}^p\right) $
are bounded as $k\to\infty$.  If $p\geq 1$ this clearly follows from Corollary
\ref{TracePowers}.
Our immediate goal is to extend this bound to $0<p<1$.  We will prove:
\begin{lemma}\label{MainLemma}
Assume $a\geq 0$ is compactly supported
and that $\Gamma$ is isotropic or co-isotropic.  Then 
for every $p\in(0,1)$ there is $C>0$ such that for all $k$
\[
k^{-d/2}\tr\left( S_{ad\sigma}^p\right) \leq C.
\]
\end{lemma}

As we will see the proof reduces to estimating the integral
\begin{equation}\label{}
\calI_p = \int_{\bbC^N}\left(\int_\Gamma e^{-k|z-w|^2}\, a(w)\, d\sigma(w)\right)^p\, dL(z)
\end{equation}
as $k\to\infty$.  

\subsubsection{Localization to a tubular neighborhood}
Let us introduce a tubular neighborhood of $\Gamma$,
\begin{equation}\label{}
\calN = \{ z\in\bbC^n\;;\; d(z,\Gamma) \leq \epsilon\}
\end{equation}
where $d(z,\Gamma) = \min_{w\in\Gamma} |z-w|$ is the distance from $z$ to
$\Gamma$, and $\epsilon>0$ is small enough so that $\calN$ is a 
bundle $\calN\to\Gamma$ whose fibers are $(2N-d)$-dimensional disks.
We will prove:
\begin{lemma}
\[
\calI_p = \int_{\calN}\left(\int_\Gamma e^{-k|z-w|^2}\, a(w)\, d\sigma(w)\right)^p\, dL(z) + 
O(k^{-\infty}).
\]
\end{lemma}
\begin{proof}
Let $\calW_N = \bbC^N\setminus\calN$ be the complement of $\calN$, and partition it
as follows,
\[
\calW_N = \calW_R\cup\calW_\infty,\qquad \calW_R = \{ z\in\calW_N\;;\; |z|\leq R\}
\]
where $R>0$ is chosen large enough so that 
\[
|z|\geq R\ \text{and}\ w\in\Gamma\quad\Rightarrow\quad |z-w| \geq |z|/2.
\]
It is clear that there is $C>0$ such that
\[
\forall z\in\calW_R,\ w\in \Gamma\quad |z-w|^2 \geq C,
\]
and therefore
\[
\int_{\calW_R}\left(\int_\Gamma e^{-k|z-w|^2}\, a(w)\, d\sigma(w)\right)^p\, dL(z)
\leq C_1 \ e^{-kpC}
\]
is negligeable.  Furthermore, for some constants $C,\, c>0$,
\[
\int_{\calW_\infty}\left(\int_\Gamma e^{-k|z-w|^2}\, a(w)\, d\sigma(w)\right)^p\, dL(z)
\leq C \int_{\{z\;;\; |z|\geq R\}} e^{-kc|z|^2}\, dL(z),
\]
and it is not hard to see that the last integral is $O(k^{-\infty})$ as well.
\end{proof}

\subsubsection{Integration over $\calN$}

Using a (finite) partition of unit we can assume without loss of generality that
the support of $a$ is contained in the image of a parametrization of $\Gamma$,
\[
\bbR^d\supset U \xrightarrow{\gamma} \Gamma.
\]
Let us introduce coordinates $s = (s_1,\ldots , s_d)\in\bbR^d$ and
$z = (x_1,\ldots , x_{2N})\in\bbR^{2N}\cong\bbC^N$.
Also introduce smooth orthonormal vector fields $E_j$, $j=1,\ldots \nu:= 2N-d$,
defined on the image of $\gamma$ which, at each point $p\in \Gamma$ span
the normal space $T_p\Gamma^\bot$.  We will regard the $E_j$ as functions
of $s$ via the parametrization.  
Letting $B(0,\epsilon) = \{ t\in\bbR^\nu\;;\; |t|<\epsilon\}$,
we can define a parametrization of $\calN$ by
\[
\begin{array}{rcl}
U\times B(0,\epsilon) & \longrightarrow & \calN\\
(s, t) & \mapsto & z(s,t) := \gamma(s) + \sum_{j=1}^\nu t_j E_j(s).
\end{array}
\]

\medskip
We now apply the method of stationary phase to the integral
\[
J_k(s,t) := \int_U e^{-k|z(s,t)-\gamma(u)|^2}\, a(\gamma(u))\, h(u)\, du,
\]
where $h(u)du = d\sigma$.  Choosing $\epsilon$ small enough we can
guarantee that the only critical point of the phase is at $u=s$, which is the
minimum of the function $u\mapsto |z(s,t)-\gamma(u)|^2$.  Since the normal moving
frame $\{E_j\}$ is orthonormal
\[
|z(s,t)-\gamma(s)|^2 = \sum_{j=1}^\nu t_j^2 ,
\]
and the method of stationary phase gives that

\begin{equation}\label{}
J_k(s,t) = \left(\frac{2\pi}{ k}\right)^{d/2}\, \frac{e^{-k|t|^2}}{\sqrt{H(s,t)}}\,
\big( a(\gamma(s))\, h(s) + O(1/k)\big),
\end{equation}
uniformly in $(s,t)$.  Here $H(s,t)$ is the determinant of the Hessian of the phase at $u=s$.
On the other hand, by the assumption on the support of $a$,
\[
\calI_p =
\int_{U\times B(0,\epsilon)} J_k(s,t)^p\, W(s,t)\, ds\, dt + O(k^{-\infty})
\]
where $W(s,t)ds\,dt = dL(z)$.
Therefore, for $k$ sufficiently large
\[
\calI_p \leq C \int_{B(0,\epsilon)} \frac{e^{-kp|t|^2}}{ k^{pd/2}}\, dt \leq C k^{-(pd+\nu)/2}.
\]
Recalling that $\nu = 2N-d$, we obtain:
\begin{lemma}
There is $C>0$ such that for all $k$ sufficiently large
\[
\calI_p \leq C \, k^{-N+d(1-p)/2}.
\]
\end{lemma}

\bigskip\noindent
{\em Proof of Lemma \ref{MainLemma}.}\ 
Let $\widetilde{T^p}$ be the Berezin transform (or Berezin symbol) of $T^p$, namely,
\[
\widetilde{T^p}(z)= \inner{T^p(k_z)}{k_z}
\]
where $k_z=e_z/\Vert e_z\Vert$ are the normalized coherent states.  Then
\[
\tr(T^p) = \constants\int_{\bbC^N}\, \widetilde{T^p}(z)\, dL(z).
\]
On the other hand,
for $0<p\leq 1$ we have that if $e$ is a unit vector then $\inner{T^p(e)}{e}\leq \inner{T(e)}{e}^p$
(see for example \cite{KZ} Proposition 1.31).  Therefore
\[
\widetilde{T^p}(z) \leq\widetilde T (z)^p\quad\text{and}\quad
\tr(T^p) \leq \constants\int_{\bbC^N}\, \widetilde{T}(z)^p\, dL(z).
\]
If we recall that
\[
\widetilde{T}(z) = \constants\, \int_\Gamma e^{-k|z-w|^2} a(w)\, d\sigma(w),
\]
we see that
\[
\tr(T^p) \leq \left(\frac{k}{\pi}\right)^{N(1+p)}\,\calI_p \leq C k^{N(1+p) - N + d(1-p)/2}
= C k^{Np+d(1-p)/2}.
\]
Since $S = 2^{-d'/2}\left(\frac{k}{\pi}\right)^{-N+d/2} T$, 
$\tr (S^p) \leq C k^{d/2}$.

\hfill$\square$

\subsection{End of the proof of Theorem \ref{Szego}}
Let us fix $a\in C_0^\infty(\Gamma)$, $a\geq 0$.  In the remainder of this section
we denote $S_{ad\sigma}$ simply by $S$.
	
We begin by introducing the functional that appears on the right-hand side
of (\ref{szegoLimit}), namely
\[
\calF(\varphi) = \int_\Gamma \calO_{-d'/2}(\varphi)(a(w))\, d\sigma(w).
\]
We will need:
\begin{lemma}\label{funcional F}
The functional $\calF$ is well-defined in the class of functions
\[
\calC = \left\{ \varphi(s)\;;\; \exists p>0\ \text{such that}\ \frac{\varphi(s)}{s^p}\ \text{is continuous on }
[0, R]\right\}.
\]
Moreover, if $\varphi(s)/s^p$ is continous on $[0, R]$
\begin{equation*}
\vert \calF(\varphi)\vert\leq C_p\Vert s^{-p}\varphi(s)\Vert_{L^\infty[0,R]}.
\end{equation*}
\end{lemma}\
\begin{proof} 
It suffices to check the case $d'>0$.
Notice that since  $s^{\mu}\vert \log(s)\vert$ if bounded in $(0,R]$, for any $\mu>0$ then $\int_0^{1}\vert s^p\log\left(1/s\right)\vert^{d'/2-1}\frac{ds}{s}$ is finite and
\begin{align*}
\int_0^{a(w)}\vert\varphi(s)\log\left(a(w)/s\right)\vert^{d'/2-1}\frac{ds}{s}&\leq\Vert t^{-p}\varphi\Vert_{L^\infty[0,R]}\int_0^{a(w)}\vert s^p\log\left(a(w)/s\right)\vert^{d'/2-1}\frac{ds}{s}\\
=&a(w)^p\Vert t^{-p}\varphi\Vert_{L^\infty[0,R]}\int_0^{1}\vert s^p\log\left(1/s\right)\vert^{d'/2-1}\frac{ds}{s}.
\end{align*}
From this the lemma follows.
\end{proof}

\medskip

\begin{lemma}\label{Crucial}
 Let $a\geq 0$ and $p>0$.
Let $f$ be a polynomial without a constant term.  Then
\[
\lim_{k\to\infty}2^{d'/2}\,\left(\frac{\pi}{k}\right)^{d/2} \tr[S^pf(S)] = \calF(t^pf(t)).
\]
\end{lemma}
\begin{proof} Let $[0, R]$ be an interval containing the spectra of $S_{ad\sigma}$ for all $k$,
and let $(g_j)$ a sequence of polynomials such that
\[
\lim_{j\to\infty} g_j(t)  =  t^{p} \quad\text{uniformly on } [0, R].
\]
We can  estimate
\begin{equation}\label{thenEstimate}
\left| 2^{d'/2}\,\left(\frac{\pi}{k}\right)^{d/2}\tr[S^pf(S)]  - \calF(t^pf(t))\right| \leq
\text{I} + \text{II} + \text{III}
\end{equation}
where
\[
\text{I} = \left| 2^{d'/2}\,\left(\frac{\pi}{k}\right)^{d/2}\tr\left[(S^p  - g_j(S))\,f(S)\right]\right|,
\]
\[
\text{II} = \left|2^{d'/2}\,\left(\frac{\pi}{k}\right)^{d/2} \tr[g_j(S)f(S)] -\calF(g_j(t) f(t)) \right|,
\]
and
\[
\text{III} = \left|\calF(g_j(t)f(t))-t^pf(t))\right|.
\]
Let $\tilde{f}$ be the polynomial whose coefficients are the absolute values
of the coefficients of $f$.  Then $ \tr\left(\left|f(S)\right|\right)\leq \tr\left(\tilde{f}(S)\right)$
and therefore, using the \Sz theorem for polynomials (or just the estimates for the traces
of powers of $S$), we have
\begin{equation}\label{hijoles}
\text{I} \leq 2^{d'/2}\,\left(\frac{\pi}{k}\right)^{d/2}\,\tr\left(\left|f(S)\right|\right)\, 
\norm{S^p-g_j(S)}\leq C\norm{S^p-g_j(S)}
\end{equation}
for some $C>0$ and all $k$.  Let $\epsilon>0$.  There exists $j$ large enough so that 
\[
\norm{S^p-g_j(S)} < \frac{\epsilon}{3C}\quad\text{and}\quad \text{III}<\frac{\epsilon}{3}.
\]
The first of these conditions together with (\ref{hijoles}) implies that $I<\epsilon/3$.
Now for each $j$ the quantity II tends to zero as $k\to\infty$, by Corollary \ref{TracePowers}.
Therefore the left-hand side of (\ref{thenEstimate}) is less than $\epsilon$ if $k$ is
large enough.
\end{proof}

\medskip\noindent
{\em End of the proof of Theorem \ref{Szego}.}
Let $\varphi\in\calC$, that is, for some $p\in(0,1)$, the function
\[
\psi(t) := \frac{\varphi(t)}{t^p}
\]
is continuous on $[0,R]$.
In the argument below $p$ can be arbitrary in $(0,1)$.  Therefore, without loss of generality we can assume that $\psi(0)=0$.

Let $(f_\ell)$ be a sequence of polynomials such that
\[
\lim_{\ell\to\infty} f_\ell(t)  = \psi(t)\quad\text{uniformly on } [0, R].
\]
Without loss of generality we can assume that $f_\ell (0)=0$ for all $\ell$.  We have that
\[
\varphi(S) = S^{p}\,\lim_{\ell\to\infty} f_\ell(S)
\]
in the operator norm
uniformly in $S$ provided $\norm{S}$ remains bounded.  
Also, by Lemma  \ref{funcional F} 
\begin{equation}\label{limite de F}
\calF (\varphi)=\lim_{\ell\rightarrow\infty}\calF (t^pf_\ell).
\end{equation}
We can now estimate:
\[
\vert\tr \left[ \varphi(S)- S^pf_\ell(S)\right]\vert = 
\vert\tr \left[ S^{p}(\psi(S)-f_\ell(S))\right] \vert\leq \tr(S^{p})\, \norm{\psi(S)-f_\ell(S)}.
\]
By Lemma \ref{MainLemma} the numbers $k^{-d/2}\tr(S^{p})$ are bounded. 
Therefore there exists $C>0$ such that 
\[
k^{-d/2}\,\vert\tr \left[ \varphi(S)- S^pf_\ell(S)\right]\vert\leq 
C\norm{\psi(S)-f_\ell(S)}\xrightarrow[\ell\to\infty]{} 0,
\]
that is,
\begin{equation}\label{uniformity}
k^{-d/2}\,\tr (\varphi(S)) = \lim_{\ell\to\infty} k^{-d/2} \tr\left[S^{p}f_\ell(S)\right]
\end{equation}
uniformly in $k$.

Applying Lemma \ref{Crucial} (which is possible since $f_\ell(0)=0$ for all $\ell$)
and exchanging the limits $\ell\to\infty$, $k\to \infty$
we get
\begin{equation}\label{}
\lim_{k\to\infty}2^{d'/2}\,\left(\frac{\pi}{k}\right)^{d/2} \tr[\varphi(S)] = \calF(\varphi(t)),
\end{equation}
as desired.
\hfill{$\Box$}

\section{Asymptotics of the Schatten norms}

When $\Gamma$ is isotropic or co-isotropic one has (\ref{muchosS}), that is
\begin{equation*}
2^{d'/2}\left(\frac{k}{\pi}\right)^{-d/2}\tr(S_{a_1d\sigma}\cdots S_{a_nd\sigma})=\frac{1}{n^{d'/2}}\int_\Gamma\prod_{j=1}^n a_j(w)\,d\sigma(w)+O(1/k).
\end{equation*}
Hence  for any polynomial $p$ such that $p(0)=0$,
\begin{equation}\label{asintvariosa}
2^{d'/2}\left(\frac{k}{\pi}\right)^{-d/2}\tr\left[p(S_{ad\sigma}^{*}S_{ad\sigma})\right]=
\frac{1}{2^{d'/2}}\int_{\Gamma}\calO_{-d'/2} (p)(\left| a(w)\right|^2)\, d\sigma + O(1/k).
\end{equation}
Recall that a bounded operator $S$ in a Hilbert space  belongs to the Schatten class $S_r$, with $r>0$ if
 $$\Vert S\Vert_r= \left(\tr(\left| S\right|^r)\right)^{1/r}<\infty,$$
 where
  $\left| S\right|=\left(S^{*}S\right)^{1/2}.$
$\Vert S\Vert_r$ is a norm for $r\geq 1,$ and for $0<r<1$
\begin{equation}\label{seminorma}
\Vert S_1 + S_2\Vert_r^{r}\leq 2\left(\Vert S_1\Vert_r^{r} + \Vert S_2\Vert_r^{r}\right).
\end{equation}
see \cite[Ch X, Lemma 9.9]{DS}.

\medskip
After these preliminary remarks we now prove Theorem \ref{Schatten}.
\begin{proof}
  For $a=a_1+ia_2$ a smooth complex function in $\Gamma$ we can write
$$S_{ad\sigma}=S_{(a_1+C)d\sigma}+iS_{(a_2+C)d\sigma}-(1+i)S_{d\sigma},$$
where $C=\Vert a\Vert_{\infty}$ so all the operators in this linear combination have positive symbols.
From \eqref{seminorma} and  Lemma \ref{MainLemma}, it follows that  $k^{-d/2} \tr\left((S_{ad\sigma}^*S_{ad\sigma})^{r/2}\right)$ is bounded in $k$  for any $r>0$. 

Denote $A=S_{ad\sigma}^*S_{ad\sigma}$. Let $\epsilon>0$ and consider
a polynomial $p_1(t)$ such that \[
\Vert t^{r/4}-p_1\Vert_{L^{\infty}[0,R]}<\epsilon/4C_1(r),
\] 
where $ C_1(r)=\sup_k  2^{d'/2}\left(\frac{k}{\pi}\right)^{-d/2}\, \tr (A^{r/4})$.
Then we have 
\begin{align*}
2^{d'/2}\left(\frac{k}{\pi}\right)^{-d/2} \left| \tr\left( A^{r/2} -A^{r/4}p_1(A)\right)\right|&= 2^{d'/2}\left(\frac{k}{\pi}\right)^{-d/2} \left| \tr\left( A^{r/4} \left( A^{r/4} -p_1(A)\right)\right) \right|\\
&\leq  2^{d'/2}\left(\frac{k}{\pi}\right)^{-d/2}\tr\left( A^{r/4}\right)\Vert A^{r/4} -p_1(A)\Vert\leq \epsilon/4.
\end{align*}
Similarly, let
$C_2=\sup_k2^{d'/2}\left(\frac{k}{\pi}\right)^{-d/2} \tr \left(\left|p_1(A)\right|\right)<\infty$ 
and let $p_2$ be a polynomial  such that
$\Vert t^{r/4}-p_2\Vert_{L^{\infty}[0,R]}<\epsilon/4C_2$. Then
 \begin{equation*}
2^{d'/2}\left(\frac{k}{\pi}\right)^{-d/2} \left| \tr\left(p_2(A)p_1(A) -A^{r/4}p_1(A)\right)\right| = 2^{d'/2}\left(\frac{k}{\pi}\right)^{-d/2}\left| \tr(p_1(A)) \left( A^{r/4} -p_2(A)\right)\right|\leq \epsilon/4. 
\end{equation*}
Thus we have found a polynomial $p=p_1p_2$ such that
\begin{equation}\label{aproxtraza}
2^{d'/2}\left(\frac{k}{\pi}\right)^{-d/2}\left|\tr \left(A^{r/2}-p(A)\right)\right|\leq \epsilon/2.
\end{equation}
Notice that we can choose the 
$p_i$ so that $\Vert t^{r/4}-p_i\Vert_{L^{\infty}[0,R]}$ is small enough to have
\begin{equation}\label{aproxint}
2^{-d'/2}\int_{\Gamma}\left|\calO_{-d'/2} (t^{r/2})-\calO_{-d'/2} (p)\right| | a(w) |^2\,  d\sigma\leq\epsilon/4 .
\end{equation}
Finally, by \eqref{asintvariosa} there exists $M$ such that if $k>M$
\begin{equation}\label{limite p}
\left| 
2^{d'/2}\left(\frac{k}{\pi}\right)^{-d/2}\tr\left[p(A)\right]-
2^{-d'/2}\int_{\Gamma}\calO_{-d'/2} (p)( |a(w)|^2)\, d\sigma \right|<\epsilon/4.
\end{equation}
 
 Since  $2^{-d'/2}\int_{\Gamma}\calO_{-d'/2} (t^{r/2})\left( |a(w)|^2\right)\,  d\sigma(w)=\int_\Gamma  \frac{\left| a(w)\right|^r}{r^{d'/2}}  d\sigma(w),$ we have by \eqref{aproxtraza},\eqref{aproxint} and \eqref{limite p} that
 \begin{equation*}
 \left| 2^{d'/2}\left(\frac{k}{\pi}\right)^{-d/2}\tr\left(A^{r/2}\right)-\int_\Gamma  \frac{\left| a(w)\right|^r}{r^{d'/2}}  d\sigma(w)\right|\leq \epsilon
\end{equation*}
 if $k>M$ and the proof is complete.
\end{proof}

\section{Estimating $\lambda_{\text{\rm max}}$
when $\Gamma$ is a Bohr-Sommerfeld Lagrangian}

In this section we obtain an asymptotic 
lower bound for the greatest eigenvalue of $T_{ad\sigma}$,
under the assumption that $a$ is real-valued and
$\Gamma$ satisfies a Bohr-Sommerfeld condition.  
We begin by recalling that the greatest eigenvalue
of $T_{ad\sigma}$ is
\begin{equation}\label{}
\lambda_{\max}(k) = \sup_{\psi\in\calB_k\setminus\{0\}}
\frac{\int_\Gamma \left|\psi(z)\right|^2\, a(z)\,d\sigma(z)}{\norm{\psi}_{\calB_k}^2}.
\end{equation}
The Bohr-Sommerfeld
condition allows us to construct a sequence $\{\psi_k\in\calB_k\;;\; k=1,2,\cdots\}$
whose micro-support is $\Gamma$, and the lower bound is obtained by considering
the asymptotics as $k\to\infty$ of
\[
\frac{\int_\Gamma \left|\psi_k(z)\right|^2\, a(z)\,d\sigma(z)}{\norm{\psi_k}_{\calB_k}^2}.
\]

In this section we work with the symplectic form on $\bbC^N$ $\Omega = -2\omega$, 
that is
\[
\Omega = i\,\sum_{j=1}^N dz_j\wedge d\zbar_j
\]
which, if we write $z_j = \frac{1}{\sqrt 2}\left( q_j-ip_j\right)$
becomes $\Omega = \sum_j dp_j\wedge dq_j$.  This rescaling of $\omega$
of course does not change the notions of isotropic/co-isotropic.

The reproducing kernel of
$\calB_k$ is now
\begin{equation}\label{}
\Pi_k(z,w) = \left(\frac{k}{\pi}\right)^N\, e^{ik\left[ -\Omega(z,w) + i|z-w|^2\right]/2}.
\end{equation}

We will also need the potential one-form
\[
\eta = \frac{i}{2}\,\sum_{j=1}^N z_j d\zbar_j - \zbar_jdz_j
\]
so that $\Omega = d\eta$.
Denote by $\iota: \Gamma\hookrightarrow\bbC^N$ the inclusion.  By the hypothesis
that $\Gamma$ is isotropic we have:
\[
d\iota^*\eta = \iota^*d\eta  = \iota^*\Omega = 0,
\]
that is, $\iota^*\eta$ is a closed one-form on $\Gamma$. It is rare that it is an
exact form.  However the following is relatively more common:

\begin{definition}
We will say that $\Gamma$ satisfies the Bohr-Sommerfeld condition iff
there is a smooth map $\Phi: \Gamma\to S^1$ such that
\begin{equation}\label{bhrS1}
\Phi^{-1}\, d\Phi = i\iota^*\eta.
\end{equation}
\end{definition}
We henceforth assume that this condition holds.  It will be convenient to introduce
the notation
\[
\Phi = e^{i\vartheta} \quad\text{with}\quad \vartheta: \widetilde\Gamma\to\bbR,
\]
where $\widetilde\Gamma\to \Gamma$ is the universal cover of $\Gamma$.
We will abuse the notation and write $\Phi(w) = e^{i\vartheta(w)}$, identifying
$\Gamma$ with a fundamental domain in $\widetilde\Gamma$.  Condition
(\ref{bhrS1}) now reads
\begin{equation}\label{bhrS2}
d\vartheta = \iota^*\eta.
\end{equation}

\begin{definition}
Let $\alpha: \Gamma\to\bbR$ be a smooth function.  
With the previous notation we let $\forall k=1,2,\cdots$
\begin{equation}\label{}
\psi_k(z) := \int_\Gamma \Pi_k(z,\wbar)\, e^{ik\vartheta(w)}\, \alpha(w)\, d\sigma(w).
\end{equation}
\end{definition}
More specifically,
\begin{equation}\label{}
\psi_k(z) = \left(\frac{k}{\pi}\right)^N\,
\int_\Gamma \ 
e^{ik\left[ -\frac 12\Omega(z,w)+\vartheta(w)+\frac i2 |z-w|^2 \right]}\, \alpha(w)\, d\sigma(w).
\end{equation}
By the compactness of $\Gamma$ it is clear that $\psi_k\in\calB_k$.  In fact,
comparing with equation (\ref{rstar}), we see that
\begin{equation}\label{}
\psi_k = \calR^*(e^{ik\vartheta(w)}\,\alpha(w)).
\end{equation}

\medskip
We can estimate the norm of $\psi_k$ as follows
in case $\Gamma$ is lagrangian (for an analogous result on compact K\"ahler manifolds see
\cite{BPU}):

\begin{lemma}  If $\Gamma$ is lagrangian,
\begin{equation}\label{}
\norm{\psi_k}^2 = \left(\frac{2k}{\pi}\right)^{N/2} \int_\Gamma |\alpha(w)|^2\, d\sigma(w)
+O(k^{N/2-1}).
\end{equation}
\end{lemma}
\begin{proof} By the reproducing property
\[
\norm{\psi_k}^2 = 
\int_{\Gamma\times\Gamma} \int_{\bbC^N}
\Pi_k(z,\wbar_1)\, e^{ik\vartheta(w_1)}\, \alpha(w_1)\, 
\Pi_k(w_2,\zbar)\, e^{-ik\vartheta(w_2)}\, \overline\alpha(w_2)\, dL(z)\,d\sigma(w_1)\, d\sigma(w_2) =
\]
\[
= \int_{\Gamma\times\Gamma}\Pi_k(w_2, w_1)\,
e^{ik(\vartheta(w_1)-\vartheta(w_2))}\, \alpha(w_1)\,
\overline\alpha(w_2)\, d\sigma(w_1)\, d\sigma(w_2) =
\]
\[
=\left(\frac{k}{\pi}\right)^{N}\int_{\Gamma}\int_\Gamma
e^{ik\left[-\frac 12 \Omega(w_2,w_1) + \vartheta(w_1)-\vartheta(w_2) + \frac i2 |w_1-w_2|^2\right]}
\, \alpha(w_1)\,\overline{\alpha}(w_2)\, d\sigma(w_1)\, d\sigma(w_2) .
\]
Now apply the method stationary phase in the inner integral (with respect to $w_1$) 
with $w_2$ fixed, that is, to
\begin{equation}\label{w1Int}
\int_\Gamma
e^{ik\left[ -\frac 12\Omega(w_2,w_1) + \vartheta(w_1) + \frac i2 |w_1-w_2|^2\right]}
\, \alpha(w_1)\, d\sigma(w_1) .
\end{equation}
In a local parametrization of $\Gamma$ $w_1 = w_1(t)$, $t\in U\subset\bbR^N$ an open set,
the derivative of the phase is 
\[
\frac{\partial\ }{\partial t_j}
\left[ \frac{1}{2i}(w_2{\wbar}_1-\wbar_2 w_1-|w_1-w_2|^2) + \vartheta(w_1)
\right]=
\]
\[
= \frac{1}{2i}(w_2\dot{\wbar}_1-\wbar_2 \dot{w}_1-\dot{w}_1(\wbar_1-\wbar_2)
-\dot{\wbar}_1(w_1-w_2) - 
w_1\dot{\wbar}_1 + \wbar_1\dot{w}_1) = i\dot{\wbar}_1(w_1-w_2).
\]
where we let $\dot{w_1} = \frac{\partial\ }{\partial t_j} w_1(t)$ for simplicity.

This shows that  the critical points of the phase in (\ref{w1Int}) are the values of $t$ such that 
\begin{equation}\label{}
\forall j=1,\ldots, d\qquad
\frac{\partial \wbar_1(t)}{\partial t_j}\cdot (w_1(t)-w_2) = 0.
\end{equation}
It is not hard to see that, in real terms,
\begin{equation}\label{}
\frac{\partial \wbar_1(t)}{\partial t_j}\cdot v = 0\quad\Leftrightarrow\quad
v\in T_{w_1}^\bot\Gamma\cap T_{w_1}^0\Gamma,
\end{equation}
where $T^\bot\Gamma$ is the metric orthogonal to $T\Gamma$ and
\[
T^\circ\Gamma = \{ w\in\bbR^{2N}\;;\; \forall u\in T\Gamma\ \Omega(u,v)=0\}
\]
is the symplectic annihilator of $T\Gamma$.  In the lagrangian case
\[
\forall w_1\in\Gamma \qquad T_{w_1}^\bot\Gamma\cap T_{w_1}^0\Gamma = 0,
\]
and therefore the only critical point of (\ref{w1Int}) is the value $t_0$ such that
$w_1(t_0)=w_2$.  The hessian matrix of the phase at the critical point is
\begin{equation}\label{}
i\begin{pmatrix}
\frac{\partial \wbar_1}{\partial t_i}\cdot\frac{\partial w_1}{\partial t_j}
\end{pmatrix}.
\end{equation}
To proceed, let us assume without loss of generality that the parametrization
of $\Gamma$ is such that the matrix $(g_{ij})$ of the metric is the identity
matrix at $t=t_0$.  Together with the assumption that $\Gamma$ is isotropic
this implies that $\begin{pmatrix}
\frac{\partial \wbar_1}{\partial t_i}\cdot\frac{\partial w_1}{\partial t_j}
\end{pmatrix}=I_{N\times N}$.  
Therefore the method of stationary phase gives that (\ref{w1Int}) equals
\begin{equation}\label{pointWise}
\left(\frac{2\pi}{k} \right)^{N/2}
e^{ik\vartheta(w_2)}\,\alpha(w_2)+ O(k^{-N/2-1}),
\end{equation}
and therefore
\begin{equation}\label{}
\norm{\psi_k}^2 = \left(\frac{2k}{\pi}\right)^{N/2} \int_\Gamma |\alpha(w)|^2\, d\sigma(w)
+O(k^{N/2-1}).
\end{equation}
\end{proof}

\bigskip
That (\ref{w1Int}) equals (\ref{pointWise}) gives the pointwise estimate
\begin{equation}\label{}
\forall z\in\Gamma\qquad
\psi_k(z) = \left(\frac{2k}{\pi}\right)^{N/2} \,e^{ik\vartheta(z)}\,\alpha(z)+ O(k^{N/2-1}),
\end{equation}
where the constants implicit in the $O$ estimate
can be taken uniformly on $z\in\Gamma$ by compactness.  Therefore
\begin{equation}\label{}
\int_\Gamma |\psi_k(z)|^2\, a(z)\, d\sigma(z) =
\left(\frac{2k}{\pi}\right)^{N}  \int_\Gamma |\alpha(z)|^2\,a(z)\, 
d\sigma(z) + O(k^{N-1}),
\end{equation}
and therefore
\begin{equation}\label{}
\frac{\int_\Gamma |\psi_k(z)|^2\,a(z)\, d\sigma(z) }{\norm{\psi_k}^2 }=
\left(\frac{2k}{\pi}\right)^{N/2} \,
\frac{\int_\Gamma |\alpha(z)|^2\,a(z)\, d\sigma(z) }
{\int_\Gamma |\alpha(w)|^2\, d\sigma(w)}
+O(k^{N/2-1}).
\end{equation}
Finally we obtain:
\begin{proposition}\label{LowerBound}  If $\lambda_{\max}(k)$ is the largest eigenvalue
of $T_{ad\sigma}$, then, for any $\alpha\in C_0^\infty(\Gamma)$ such 
that $\norm{\alpha}_{L^2}=1$
\begin{equation}\label{}
\lambda_{\max}(k)\geq 
\left(\frac{2k}{\pi}\right)^{N/2} \,
\int_\Gamma |\alpha(z)|^2\,a(z)\, d\sigma(z) 
+O(k^{N/2-1}).
\end{equation}
\end{proposition}

\noindent{\em Remarks:}
\begin{enumerate}
\item In particular, if $a\equiv 1$ the asymptotic lower bound obtained is simply
$\left(\frac{2k}{\pi}\right)^{N/2} $.  It is universal (independent of $\Gamma$).
\item If we take $\alpha$ to be constant,  by virtue of Theorem (\ref{NormEst}) we can
conclude that $\exists C,\, C'>0$ such that
\begin{equation}\label{}
C\Vert a\Vert_{\infty}\,k^{N/2}\geq \lambda_{\max}(k) \geq C'\, k^{N/2}\,\inf(a).
\end{equation}
\end{enumerate}

\begin{appendix}
\section{Computing $\det$(Hessian)}

We present here the computation of the determinant of the matrix $S$ equal to
$-\sqrt{-1}$ times the Hessian, that is, the matrix (\ref{theHessian}).  
For convenience we write $q=n-1$.  The matrix (\ref{theHessian}), partitioned into 
a $q\times q$ array of $d\times d$ blocks, is equal to
\[
S_q =
\begin{pmatrix}
G & 0 &0 &\cdots &0\\
 0& G & 0&\cdots &0\\
0 & \ddots & \ddots &\ddots & 0\\
0 & \cdots & 0 &G & 0\\
0 & \cdots & \cdots &0 & G
\end{pmatrix}\times M_q
\]
where $M_q$ is the block tri-diagonal Toeplitz matrix
\[
M_q = 
\begin{pmatrix}
2I & \overline{B} &0 &\cdots &0\\
B & 2I & 0&\cdots &0\\
0 & \ddots & \ddots &\ddots & 0\\
0 & \cdots & B &2I & \overline{B}\\
0 & \cdots & \cdots & B & 2I
\end{pmatrix}
\]
with
\[
B := -I + iG^{-1}H.
\]
All blocks consist of $d\times d$ matrices.
Since $H$ is skew-symmetric $B$ is Hermtian and $M_q$ is symmetric.
We will prove that $\det(M_q) = \Delta_n^2$.

\medskip
In the calculation of $\det(S_q)$, we follow the approach of \cite{JRS}.
Let $\calR$ be the ring of $d\times d$ complex matrices generated by the identity
and the matrix
\[
W:=G^{-1}H.
\]
Clearly $\calR$ is a commutative ring (in particular $[B, \overline{B}\,]=0$).  The matrix $M_q$ is a
$q\times q$ matrix with entries in $\calR$.  Any such matrix has a determinant,
which we denote with a capital $D$, that is an element in $\calR$; in particular
\[
\Det(M_q)\in \calR.
\]
$\Det$ is defined by the usual formula which is unambigous since $\calR$ is commutative. 
All the usual rules for computing determinants carry over to computing $\Det$, and one
has the theorem that for any $d\times d$ matrix $L$ with entries in $\calR$,
its numerical determinant equals
\[
\det (L) = \det(\Det(L)).
\]
We will use this result to compute $\det(M_q)$, first recursively and later in closed form.  Since 
\[
\det(S_q) = \det(G)^q\,\det(M_q),
\]
we will obtain a formula for $\det(S_q)$.

\begin{lemma}\label{detCalculation}
Let
\[
Z = B\overline{B} = I+W^2.
\]
Then one has
\begin{equation}\label{obvio1}
\Det(M_1) = 2I,
\end{equation}
\begin{equation}\label{obvio2}
\Det(M_2) = 4I - Z = 3I-W^2,
\end{equation}
and for each $q=2,3,\ldots$
\begin{equation}\label{induccion}
\Det(M_{q+1}) = 2\Det(M_q) - Z\Det(M_{q-1}).
\end{equation}
\end{lemma}
\begin{proof}
(\ref{obvio1}) and (\ref{obvio2}) are immediate, and (\ref{induccion}) is obtained
by expanding $\Det(M_{q+1})$ along the top row.
\end{proof}

\begin{corollary}
\[
\det(S_1) = 2^d\det(G),
\]
and
\[
\det(S_2) = \det(G)^2\times \det(3I-W).
\]
\end{corollary}

\bigskip
It is possible to solve (\ref{induccion}) in closed form, as follows.  Let us write, for simplicity of
notation,
\[
D_q = \Det(M_q).
\]
The recursion relation $D_{q+1} = 2D_q - ZD_{q-1}$ can itself be
written in matrix form
\[
\begin{pmatrix}
D_q\\D_{q+1}
\end{pmatrix} = 
\begin{pmatrix}
0 & I\\
-Z & 2I
\end{pmatrix}\;
\begin{pmatrix}
D_{q-1}\\D_{q}
\end{pmatrix},
\]
and therefore,
introducing the matrix with entries in $\calR$
\[
T:= 
\begin{pmatrix}
0 & I\\
-Z & 2I
\end{pmatrix},
\]
we see that
\begin{equation}\label{recursivoMat}
\begin{pmatrix}
D_{q-1}\\D_{q}
\end{pmatrix} = 
T^{q-2}
\begin{pmatrix}
D_{1}\\D_{2}
\end{pmatrix} =
T^{q-2}
\begin{pmatrix}
2I\\ 3I-W^2
\end{pmatrix}.
\end{equation}
Let us diagonalize $T$ to compute its powers.  $\Det(T) = Z$, and
the ``eigenvalues" in $\calR$ of this matrix
are
\[
\Lambda_{1,2} := I \mp\sqrt{I-Z} = I\mp i W.
\]
One can check that the column vectors of 
\[
S = \begin{pmatrix} 
I & I \\
\Lambda_1 & \Lambda_2
\end{pmatrix}
\]
are eigenvectors of $T$.  More precisely
$
S^{-1} T S = 
\begin{pmatrix}
\Lambda_1 & 0\\
0 & \Lambda_2
\end{pmatrix},
$
and therefore
\[
T^q = S
\begin{pmatrix}
\Lambda_1^q & 0\\
0 & \Lambda_2^q
\end{pmatrix}
S^{-1}.
\]
A computation shows that
\[
S
\begin{pmatrix}
\Lambda_1^q & 0\\
0 & \Lambda_2^q
\end{pmatrix} =
\begin{pmatrix}
\Lambda_1^q & \Lambda_2^q\\
\Lambda_1^{q+1} & \Lambda_2^{q+1} 
\end{pmatrix}
\]
while, since $\Det(S) = \Lambda_2-\Lambda_1 = 2iW$,
\footnote{Assuming $W= G^{-1}H$ is invertible in $\calR$.  
Otherwise, perturb $W$ a little so that it is
invertible.  The final result will be a polynomial expression for $D_q$ in terms of $W$, 
so it will remain valid
in case $W$ is not invertible, by continuity.}
\[
S^{-1} = (2iW)^{-1} 
\begin{pmatrix}
\Lambda_2 & -I\\
-\Lambda_1 & I
\end{pmatrix}.
\]
Combining, we obtain
\[
T^q =  (2iW)^{-1} 
\begin{pmatrix}
\ast & \ast\\
\Lambda_1^{q+1}\Lambda_2 -\Lambda_2^{q+1}\Lambda_1 &
-\Lambda_1^{q+1} + \Lambda_2^{q+1}
\end{pmatrix}
\]
where the computation of the starred entries is not necessary, since we are only
interested in the equation for the second component of (\ref{recursivoMat}).

Let us now compute the bottom row of $T^q$.  Recalling that $\Lambda_1\Lambda_2 = Z$, 
\[
\Lambda_1^{q+1}\Lambda_2 -\Lambda_2^{q+1}\Lambda_1 = Z(\Lambda_1^q - \Lambda_2^q).
\]
Therefore
\begin{align}\label{align}
2iW\,\Det(M_{q+2}) &= 2(I+W^2)(\Lambda_1^q-\Lambda_2^q) - 
(3I-W^2)(\Lambda_1^{q+1} - \Lambda_2^{q+1}) =\nonumber \\
& = \Lambda_1^q\left(2I + 2W^2 - (3I-W^2)(I-iW)\right) -
\Lambda_2^q\left(2I + 2W^2 - (3I-W^2)(I+iW)\right).
\end{align}
The factor of $\Lambda_1^q$ in this expression is
\[
2I+2W^2-3I + 3iW +W^2-iW^3 = \left(iW-I\right)^3 = -\Lambda_1^3.
\]
Similarly, the factor of $-\Lambda_2^q$ in (\ref{align}) is
\[
2I+2W^2 - 3I - 3iW + W^2 + iW^3 = -(iW+I)^2 = -\Lambda_2^3.
\]
Substituting back into (\ref{align}) (and shifting the value of $q$) we obtain
\begin{equation}\label{align2}
2iW\,\Det(M_{q}) = -\Lambda_1^{q+1} + \Lambda_2^{q+1}.
\end{equation}

Now $\Lambda_{1,2} = I\mp iW$, therefore
\[
\Lambda_1^{q+1} - \Lambda_2^{q+1} = (I-iW)^{q+1} - (I+iW)^{q+1}.
\]
The right-hand side is a polynomial in $W$ without constant term.
Substituting back into (\ref{align2}) concludes the proof of:
\begin{proposition}\label{ClosedForm}
\[
\Det(M_q) = i\,\frac{(I-iW)^{q+1} - (I+iW)^{q+1}}{2W} = 
\sum_{j=0}^{\floor{\frac{q}{2}}} {q+1\choose 2j+1}(-1)^j\, W^{2j}.
\]
\end{proposition}

Next we show:
\begin{lemma}
For each $t$, 
$W(t)$ is the matrix of the transformation
\[
K:= \left(\Pi_{\gamma(t)}\circ J\right) : T_{\gamma(t)}\Gamma\to T_{\gamma(t)}\Gamma
\]
in the basis $\{\gamma_i(t)\}$.
\end{lemma}
\begin{proof}
The definition of the matrix $W$
is equivalent to $GW=H$, where
\[
G = \Large( \gamma_i\cdot\gamma_j\Large)\quad\text{and}\quad 
\quad H = \Large(\omega(\gamma_i,\gamma_j)\Large).
\]
Since $\omega(v, w) = v\cdot J(w)$,
we can re-write this in the form
$
\forall i,\, k\  \sum_j G_{ij}W_{jk} = H_{ik},
$
or
\[
\forall i,\, k\qquad  \gamma_i\cdot\sum_j W_{jk}\, \gamma_j = \gamma_i\cdot J(\gamma_k).
\]
Since $\{\gamma_i\}$ is a basis of $V$, this is equivalent to
$
\Pi J(\gamma_k) = \sum_j W_{jk}\, \gamma_j.
$
\end{proof}

\medskip
As we already saw in \textsection 1, 
at any tangent space $V$ of $\Gamma$
$K$ is skew-adjoint.  We have denoted its non-zero eigenvalues (with multiplicities) as
\[
\pm\,i\lambda_\ell,\quad 0<\lambda_1\leq  \cdots\leq \lambda_r.
\]
Therefore the eigenvalues of $W^{2j}$ are $(-1)^j\lambda_\ell^2$, each with
double the multiplicity as eigenvalues of $K$, together with the eigenvalue zero with multiplicity $d-2r$.
Therefore, by (\ref{ClosedForm}) the eigenvalues of $\Det(M_q)$ are 
\begin{equation}\label{alternatively}
\sum_{j=0}^{\floor{\frac{q}{2}}} {q+1\choose 2j+1} \lambda_\ell^{2j} = 
\frac{(1+\lambda_\ell)^{q+1} - (1-\lambda_\ell)^{q+1}}{2\lambda_\ell}\qquad
\ell =1,\ldots , r,
\end{equation}
each with double the multiplicity, together with the eigenvalues corresponding to the kernel
of $W$, namely $(q+1)$ with multiplicity $d-2r$.

We can then conclude that
\begin{equation}\label{laRaizCuadrada}
\sqrt{\det(M_q)} = (q+1)^{\frac d2 - r}\,
\prod_{\ell=1}^r\,\frac{(1+\lambda_\ell)^{q+1} - (1-\lambda_\ell)^{q+1}}{2\lambda_\ell}
= \Delta_{q+1}.
\end{equation}

\end{appendix}


\begin{thebibliography}{99}

\bibitem{Be} F. A.  Berezin, Covariant and contravariant symbols of operators.  
Izv. Akad. Nauk SSSR Ser. Mat. {\bf 36} (1972) no. 5, 1134--1167 (in Russian). English
translation: USSR Izv. {\bf 6} (no. 5) (1972).

\bibitem{BMS} M. Bordemann, E. Meinrenken and M. Schlichenmaier, 
Toeplitz quantization of KŠhler manifolds and gl(N), $N\to\infty$ limits.
Comm. Math. Phys. {\bf 165} , no. 2, 281"1¤7296 (1994).

\bibitem{BPU} D. Borthwick, T. Paul and A. Uribe, 
Legendrian distributions with applications to relative Poincar\'e series. 
Invent. Math. {\bf 122} (1995), no. 2, 359"1¤7402. 

\bibitem{BG} L. Boutet de Monvel, V. Guillemin, {\sl The spectral theory of toeplitz operators}, Annals of Mathematics Studies, vol. 99. Princeton University Press, Princeton (1981)


\bibitem {DS} N. Dunford and J. T. Schwartz, \textit{Linear operators Part II}, Reprint of the 1963 original. Wiley Classics Library. A Wiley-Interscience Publication. John Wiley and Sons, Inc., New York, 1988. 

\bibitem{BH} Brian Hall, {\em Quantum mechanics for mathematicians},
Graduate Texts in Mathematics {\bf 267}. Springer, New York, 2013.

\bibitem{GUW} V. Guillemin, A. Uribe and Z. Wang, Semiclassical states associated to isotropic submanifolds of phase space. {\sl Lett Math Phys} (2016). doi:10.1007/s11005-016-0853-7.

\bibitem{Hor} L. H\"ormander, {\em The analysis of linear partial differential operators. I. Distribution theory and Fourier analysis.}  Springer-Verlag, Berlin, 2003.

\bibitem{LF} Y. Le Floch, Bounds for fidelity of semiclassical Lagrangian states in 
K\"ahler quantization.  arXiv:1705.01374.

\bibitem{JRS} J. R. Silvester, Determinants of block matrices,  Math. Gaz.
{\bf 84} no. 501 (2000), pp. 460 -- 467.

\bibitem{KZ} K. Zhu, {\em Operator theory in function spaces},
 Mathematical Surveys and Monographs {\bf 138}. AMS, Providence, RI, 2007

\end{thebibliography}
\end{document}